\documentclass[onefignum,onetabnum]{siamonline250211}

\headers{Conditionals Improve Conditioning}{Chowdhary, Milinanni, Chung, and Newman}
\title{Exploiting Exact Conditionals Improves Conditioning: \\ Provably Fast Mixing Time Bounds By Sampling from the Marginal}
\author{Abhijit Chowdhary\thanks{Department of Mathematics, Tufts University
  (corresponding author: \email{abhijit.chowdhary@tufts.edu};
  \email{e.newman@tufts.edu}).}
\and Federica Milinanni\thanks{Department of Industrial Engineering and Management
  Sciences, Northwestern University
  (\email{federica.milinanni@northwestern.edu}).}
\and Julianne Chung\thanks{Department of Mathematics, Emory University
  (\email{jmchung@emory.edu}).}
\and Elizabeth Newman\footnotemark[1]}

\usepackage{amsfonts,amssymb}
\usepackage{mathrsfs}
\usepackage{mathtools}
\usepackage{bm}
\usepackage{algorithm} %
\usepackage{algpseudocodex}
\usepackage{comment}
\usepackage[sort]{cite}
\usepackage{setspace}
\usepackage{makecell}
\usepackage{subcaption}
\makeatletter
\@ifpackageloaded{cleveref}{}{\usepackage{cleveref}}
\makeatother

\RequirePackage{multirow}
\RequirePackage{adjustbox}
\RequirePackage{colortbl}
\RequirePackage{import}
\RequirePackage{stackrel}
\RequirePackage{wrapfig}
\RequirePackage{xspace} %
\RequirePackage{bm}
\makeatletter
\@ifpackageloaded{hyperref}{}{\RequirePackage{hyperref}}
\makeatother

\newcommand{\thf}  {\tfrac{1}{2}}                            %
\renewcommand{\t} {^{\top}}                                %
\renewcommand{\d} {{\rm d}}                                %

\newcommand{\bigO}  [1]  {\calO\!\left(#1\right)}
\newcommand{\isocond}{\hat\kappa}

\newcommand{\bfLambda}{{\boldsymbol{\Lambda}}}

\newcommand{\bfSigma}{{\boldsymbol{\Sigma}}}

\newcommand{\bfzeta}{{\boldsymbol{\zeta}}}

\newcommand{\bfmu}{{\boldsymbol{\mu}}}

\newcommand{\bfB}{{\bf B}}

\newcommand{\bfF}{{\bf F}}

\newcommand{\bfH}{{\bf H}}
\newcommand{\bfI}{{\bf I}}

\newcommand{\bfS}{{\bf S}}

\newcommand{\bfX}{{\bf X}}
\newcommand{\bfY}{{\bf Y}}
\newcommand{\bfZ}{{\bf Z}}

\newcommand{\bfb}{{\bf b}}

\newcommand{\bfd}{{\bf d}}
\newcommand{\bfe}{{\bf e}}

\newcommand{\bfr}{{\bf r}}

\newcommand{\bfx}{{\bf x}}
\newcommand{\bfy}{{\bf y}}
\newcommand{\bfz}{{\bf z}}
\newcommand{\bfzero}{{\bf0}}

\newcommand{\calN}{\mathcal{N}}
\newcommand{\calO}{\mathcal{O}}

\newcommand{\bbC}{\mathbb{C}}

\newcommand{\bbR}{\mathbb{R}}

\newcommand{\bbZ}{\mathbb{Z}}

 \usepackage{enumitem}
\newsiamremark{remark}{Remark}
\newsiamremark{assumption}{Assumption}

\crefname{assumption}{assumption}{assumptions}
\Crefname{assumption}{Assumption}{Assumptions}
\AddToHook{env/assumption/begin}{\crefalias{theorem}{assumption}}

\usepackage{multirow}
\usepackage{booktabs}
\usepackage{tcolorbox}

\usepackage{tikz}
	\usetikzlibrary{calc}
	\usetikzlibrary{math}
	\usetikzlibrary{positioning}
	\usetikzlibrary{shapes.geometric}
	\usetikzlibrary{arrows.meta}
    	\usetikzlibrary{patterns}
	\usetikzlibrary{fit}

\usepackage{pgfplots}
    \pgfplotsset{compat=1.18}
    \usepgfplotslibrary{groupplots}
    \usepgfplotslibrary{fillbetween}
    \usepgfplotslibrary{colormaps}
\usepackage{pgfplotstable}
\usepackage{booktabs}
\usepackage{array}
\usepackage{colortbl}

\pgfdeclarelayer{background layer}
\pgfdeclarelayer{foreground layer}
\pgfsetlayers{background layer,main,foreground layer}

 \usepackage{pgfplotstable}
 \usepackage{etoolbox}

\usepackage{amsmath}
\usepackage{amsfonts}
\usepackage{mathtools}
\usepackage{bbm}

\usepackage{graphicx}
\usepackage{wrapfig}
\usepackage[dvipsnames]{xcolor}

\def\mydefb#1{\expandafter\def\csname bf#1\endcsname{\mathbf{#1}}}
\def\mydefallb#1{\ifx#1\mydefallb\else\mydefb#1\expandafter\mydefallb\fi}
\mydefallb aAbBcCdDeEfFgGhHiIjJkKlLmMnNoOpPqQrRsStTuUvVwWxXyYzZ\mydefallb

\def\mydefb#1{\expandafter\def\csname #1bb\endcsname{\mathbb{#1}}}
\def\mydefallb#1{\ifx#1\mydefallb\else\mydefb#1\expandafter\mydefallb\fi}
\mydefallb aAbBcCdDeEfFgGhHiIjJkKlLmMnNoOpPqQrRsStTuUvVwWxXyYzZ\mydefallb

\def\mydefb#1{\expandafter\def\csname #1tt\endcsname{\texttt{#1}\xspace}}
\def\mydefallb#1{\ifx#1\mydefallb\else\mydefb#1\expandafter\mydefallb\fi}
\mydefallb aAbBcCdDeEfFgGhHiIjJkKlLmMnNoOpPqQrRsStTuUvVwWxXyYzZ\mydefallb

\def\mydefb#1{\expandafter\def\csname #1cal\endcsname{\mathcal{#1}}}
\def\mydefallb#1{\ifx#1\mydefallb\else\mydefb#1\expandafter\mydefallb\fi}
\mydefallb aAbBcCdDeEfFgGhHiIjJkKlLmMnNoOpPqQrRsStTuUvVwWxXyYzZ\mydefallb

\def\mydefgreek#1{\expandafter\def\csname bf#1\endcsname{\text{$\boldsymbol{\csname #1\endcsname}$}}}
\def\mydefallgreek#1{\ifx\mydefallgreek#1\else\mydefgreek{#1}%
   \lowercase{\mydefgreek{#1}}\expandafter\mydefallgreek\fi}
\mydefallgreek {alpha}{Alpha}{beta}{Beta}{gamma}{Gamma}{delta}{Delta}{epsilon}{Epsilon}{zeta}{Zeta}{eta}{Eta}{theta}{Theta}{iota}{Iota}{kappa}{Kappa}{lambda}{Lambda}{mu}{Mu}{nu}{Nu}{omicron}{Omicron}{pi}{Pi}{rho}{Rho}{sigma}{Sigma}{tau}{Tau}{upsilon}{Upsilon}{phi}{Phi}{xi}{Xi}{chi}{Chi}{psi}{Psi}{omega}{varepsilon}{varphi}{Omega}\mydefallgreek

\def\mydefb#1{\expandafter\def\csname T#1\endcsname{\boldsymbol{\mathcal{\MakeUppercase{#1}}}}}
\def\mydefallb#1{\ifx#1\mydefallb\else\mydefb#1\expandafter\mydefallb\fi}
\mydefallb aAbBcCdDeEfFgGhHiIjJkKlLmMnNoOpPqQrRsStTuUvVwWxXyYzZ\mydefallb

\usepackage{xspace}

\newcommand{\nx}{{n_{x}}}
\newcommand{\ny}{{n_{y}}}

\newcommand{\pjoint}{\pi_{\rm joint}}
\newcommand{\pcond}{\pi_{\rm cond}}
\newcommand{\pmarg}{\pi_{\rm marg}}

\NewDocumentCommand{\warmstart}{m g}{%
  \IfNoValueTF{#2}
    {%
        \beta(#1)
    }
    {%
        \beta(#1 \Vert #2)
    }%
}

\newcommand{\ESS}{\mathrm{ESS}}
\newcommand{\MALA}{\textnormal{\texttt{MALA}}\xspace}

\newcommand{\MarCo}{\textnormal{\texttt{MarCo}}\xspace}
\newcommand{\Joint}{\textnormal{\texttt{JointMH}}\xspace}
\newcommand{\Oneblock}{\textnormal{\texttt{OneBlock}}\xspace}
\newcommand{\jointMH}{\textnormal{\texttt{joint MH}}\xspace}

\newcommand{\OneblockMCMC}{\textnormal{\texttt{OneBlockMH}}\xspace}

\usepackage[dvipsnames]{xcolor}
\usepackage[textsize=scriptsize]{todonotes}
\newcommand{\liznote}[1]{\todo[inline,linecolor=cyan,backgroundcolor=cyan!25,bordercolor=cyan]{Liz: {#1}}}

\definecolor{matplotlib0}{HTML}{1f77b4}
\definecolor{matplotlib1}{HTML}{ff7f0e}
\definecolor{matplotlib2}{HTML}{2ca02c}
\definecolor{matplotlib3}{HTML}{d62728}
\definecolor{matplotlib4}{HTML}{9467bd}
\definecolor{matplotlib5}{HTML}{8c564b}
\definecolor{matplotlib6}{HTML}{e377c2}
\definecolor{matplotlib7}{HTML}{7f7f7f}
\definecolor{matplotlib8}{HTML}{bcbd22}
\definecolor{matplotlib9}{HTML}{17becf}

\pgfplotstableset{
        create on use/samplername/.style={
        create col/set list={\Joint, \Oneblock, \MarCo}},
        columns/samplername/.style={column type=c, string type, column name={}}
        }

\pgfplotstableset{
    metricstablestyle/.style={
    	every head row/.style={before row=\toprule, after row=\midrule},
	every last row/.style={after row=\bottomrule},
        every row 0 column 0/.style={
            postproc cell content/.append style={
                /pgfplots/table/@cell content/.add={\cellcolor{matplotlib0!25}}{},
            }
        },
        every row 1 column 0/.style={
            postproc cell content/.append style={
                /pgfplots/table/@cell content/.add={\cellcolor{matplotlib2!25}}{},
            }
        },
        every row 2 column 0/.style={
            postproc cell content/.append style={
                /pgfplots/table/@cell content/.add={\cellcolor{matplotlib1!25}}{},
            }
        },
}
}

\begin{document}

\maketitle

\begin{abstract}
The problem of sampling from a probability distribution arises in many applications such as posterior sampling in hierarchical Bayesian inverse problems and Gaussian processes for machine learning.
Markov chain Monte Carlo (MCMC) algorithms are often used for sampling from a target probability distribution, but implementations can be computationally expensive, especially for large-scale problems. %
In certain applications, the target distribution naturally factorizes into a lower dimensional marginal distribution and a conditional distribution that allows exact sampling.
We describe an MCMC algorithm called \MarCo that exploits such a structure and generates a Markov chain via Metropolis-Hastings sampling from the marginal distribution, followed by sampling from the exact conditional distribution.
By design, \MarCo constructs a Markov chain on the joint space that inherits the convergence behavior of the marginal MCMC algorithm.
This provides multiple theoretical and computational advantages.
We prove that \MarCo can achieve improved mixing time upper bounds compared to direct sampling from the joint distribution. %
Moreover, compared to one-block methods that also exploit marginal-conditional structure, we use the framework of Peskun-Tierney ordering to show that \MarCo has a larger right spectral gap and smaller asymptotic variance, thus leading to %
superior convergence properties.
Numerical results illustrate the performance benefits of \MarCo and are provided for various problems, including a semi-blind image deblurring example.
\end{abstract}

\begin{keywords}
MALA, marginals, Metropolis--Hastings, mixing time, spectral gap, conditioning
\end{keywords}

\begin{MSCcodes}
60J22, 65C05, 60B10
\end{MSCcodes}

\tableofcontents

\section{Introduction}

Sampling from a distribution is a foundational problem in computational statistics \cite{andrieu_introduction_2003,ghanem_bayesian_2017}.
Many distributions, for instance those arising from a Bayesian inverse problem or a hierarchical model, do not have analytical expressions for their densities that would allow efficient, exact sampling.
However, they usually have access to evaluation up to a normalizing constant.
In this setting, one turns to Markov chain Monte Carlo (MCMC) methods~\cite{hastings_monte_1970,metropolis_equation_1953,robert_monte_2004}.
These methods construct a Markov chain whose stationary distribution is the target distribution and produce approximate samples once the chain has mixed; that is, has reached stationarity.

A central difficulty in MCMC is that mixing can be prohibitively slow.
Therefore, theoretically assessing the number of steps required to reach stationarity, the \emph{mixing time}, has long been a point of interest in the sampling community \cite{brown_convergence_2024}.
In recent years, there has been much progress in quantifying nonasymptotic upper bounds to mixing times.
These upper bounds are often related to polynomial factors of the dimension of the distribution and its curvature \cite{altschuler_faster_2024,wu_minimax_2022,chewi_analysis_2025,altschuler_resolving_2022,dwivedi_log-concave_2019}.
In addition, the community has developed the classical analysis of the asymptotic properties of the underlying Markov processes, such as the spectral gap and asymptotic variance~\cite{douc_markov_2018,andrieu_explicit_2024}, or comparison results such as the Peskun-Tierney ordering~\cite{peskun_optimum_1973,tierney_markov_1994}.
These tools allow us to understand what procedures may improve mixing and enable the analysis and benchmarking of different algorithms.

With an eye toward algorithmic design, we question if there exist classes of target distributions whose structure can be naturally exploited in light of the algorithms' mixing time bounds.
To that end, let us consider probability distributions in which the input variables split into two blocks, i.e., $\pjoint = \pjoint(\d \bfx, \d \bfy)$, where $(\bfx, \bfy) \in \bbR^{\nx} \times \bbR^{\ny}$.
Direct application of MCMC without exploiting this block structure can result in joint samplers that exhibit very slow mixing times.

The community has developed various methods that do exploit block structure.
The most common approach is Gibbs sampling, which alternates exact draws from conditional distributions \cite{geman_stochastic_1984,gelfand_sampling-based_1990,gelman_bayesian_2013}.
Gibbs-type schemes are the workhorses of hierarchical Bayesian computation, latent variable models, and graphical models because updating one block at a time can be far cheaper than sampling from the joint directly.
However, the computational cost of exact sampling from the conditionals can become prohibitive for large-scale problems.
Moreover, their mixing still depends on the coupling between blocks, and slow mixing can persist.

Various approaches have been considered to address these drawbacks.
For example, marginalization-based sampling approaches have been developed to exploit the marginal-conditional structure that comes from factorizing a joint density $\pjoint$ over $\bbR^\nx \times \bbR^\ny$ into
	\begin{align}\label{eq:joint_factorization}
		\pjoint(\d \bfx, \d \bfy) = \pcond(\d \bfx \mid \bfy) \pmarg(\d \bfy)
	\end{align}
where $\pcond(\d \bfx \mid \bfy)$ is the conditional density over $\Rbb^\nx$ and the marginal density over $\Rbb^\ny$ is
	\begin{align}
    \label{eq:marginal_y}
		\pmarg(\d \bfy) = \int_{\Rbb^\nx} \pjoint(\bfx, \d\bfy)\, \d\bfx.
	\end{align}
An approach based on such a factorization is the one-block algorithm~\cite{rue_gaussian_2005}. 
This constructs a Markov chain by proposing $\bfy^\star$ from a proposal distribution on $\bbR^\ny$, generating $\bfx^\star$ by exact sampling from $\pcond(\d\bfx\mid\bfy^\star)$, and implementing the Metropolis-Hastings acceptance/rejection step on the proposal  $(\bfx^\star,\bfy^\star)$.
In \cite{van_dyk_marginal_2010,papaspiliopoulos_retrospective_2008,liu_collapsed_1994,rue_gaussian_2005}, a \emph{collapsed} or \emph{marginal} algorithm was used, where $\bfx$ was integrated out entirely and MCMC methods are applied directly to the marginal (i.e., in the reduced space $\bbR^{\ny}$).
A numerical study of a marginal-then-conditional sampler was considered in \cite{fox_fast_2016} for image deblurring, where an MCMC method was applied to the marginal directly and, after obtaining effectively independent samples, draws from the conditional are made.
In \cite{saibaba_efficient_2019}, marginalization-based techniques that combine low-rank approximations with delayed acceptance and pseudomarginalization techniques were considered.
Moreover, leveraging this perspective and recent theory for MCMC \emph{mixing-time} guarantees, we conduct a rigorous analysis that quantifies the advantage of the \MarCo algorithm.
In particular,
\begin{enumerate}
\item We show that \MarCo inherits the convergence properties of the marginal sampling and generates a Markov chain that converges to the target joint distribution in total variation (\Cref{sec:marco_convergence}).
\item Since \MarCo inherits the mixing behavior of the marginal, we relate the curvature of the marginal and the joint to show that \MarCo converges to the joint distribution faster than existing approaches (\Cref{sec:theory_marco_vs_joint}).
\item Using a Peskun-Tierney ordering argument, we show that \MarCo exhibits a lower asymptotic variance and a larger right spectral gap, and thereby attains faster asymptotic convergence, than one-block methods (\Cref{sec:theory_marco_vs_one_block}).
\end{enumerate}

We remark that there are many examples of distributions which satisfy the above marginal-conditional structure and assumptions.
A rich and widely applicable example of this setting is when $\pcond(\d \bfx \mid \bfy)$ is Gaussian.
This often arises in hierarchical Bayesian inverse problems~\cite{saibaba_efficient_2019} and linear mixed-effects models~\cite{domke_hamiltonian_2024}.
Although these provide a nice testbed of examples, our approach and the corresponding theory apply to more general distributions.

\paragraph{Outline} In \Cref{sec:background}, we provide an overview of the Metropolis-Hastings (MH) algorithm, including tools for describing convergence behavior of Markov chains.
Then in \Cref{sec:methodology}, we describe \MarCo, providing some theoretical justifications and algorithmic comparisons to existing methods for sampling in the joint space.
The main theoretical results of the paper are provided in \Cref{sec:theory_marco_vs_joint,sec:theory_marco_vs_one_block}, where we provide both nonasymptotic and asymptotic convergence results for \MarCo.
Numerical results illustrating the theory are provided in \Cref{sec:numerics}, and conclusions and future directions are discussed in \Cref{sec:conclusions}.

\section{Background}
\label{sec:background}
Let $\mathcal{P}(\Rbb^d)$ denote the set of probability measures on $\Rbb^d$. 
Our goal is to draw samples from a target probability distribution, $\pi(\d \bfz) \in \Pcal(\Rbb^d)$. Throughout this paper, we consider target probability distributions that are absolutely continuous with respect to the Lebesgue measure on $\Rbb^d$. 
We denote sampling from a distribution by $\bfz \sim \pi$, $\bfz\sim \pi(\cdot)$, or  $\bfz \sim \pi(\d \bfz)$, and denote the corresponding probability density function as $\pi(\bfz)$; i.e., $\pi(\d\bfz)=\pi(\bfz)\d\bfz$. All probability densities are defined with respect to the Lebesgue measure, $\d\bfz$. %

In practice, sampling directly from a target distribution and evaluating the normalizing constant of a probability density function can be difficult. 
These challenges have motivated the development of MCMC methods to draw approximate samples from complex distributions using evaluations of unnormalized density functions \cite{meyn_markov_1993, robert_monte_2004, craiu_handbook_2026}.
A (discrete-time) Markov chain is a sequence of random variables $(\bfZ^0, \bfZ^1, \dots)$ that defines a stochastic process satisfying the Markov property. 
The evolution of a Markov chain is governed by a transition kernel, $k: \Rbb^d \times \Bcal(\Rbb^d) \to [0,1]$, where $\Bcal(\Rbb^d)$ is the Borel $\sigma$-algebra over $\Rbb^d$. 
Specifically, given a state realization at iteration $t$, $\bfZ^t = \bfz^t$, the distribution of the next random variable in the chain, $\bfZ^{t+1}$, is given by $k(\bfz^t, \cdot)$. 
As a result, an initial distribution, $\pi^0$, paired with the transition kernel, $k$, defines a sequence of probability distributions, $\pi^t$, where $\pi^t$ is the distribution of $\bfZ^t$. Formally, $\pi^t(\d\bfz)=\int k^{t}(\hat\bfz,\d\bfz)\pi^0(\d\hat\bfz)$, where $k^{t}$ is the $t$-th iterated transition kernel.
We say a sample path of length $T$ is a realization of a Markov chain $\{\bfz^t\}_{t=0}^T$, where $\bfz^0\sim\pi^0$ and $\bfz^{t+1}\sim k(\bfz^t,\cdot)$ for each $t=0,\dots,T-1$.

\subsection{Convergence of Markov Chains}
\label{sec:convergence_markov_chains}
To formalize the concept of convergence of Markov chains, we first define a notion of distance between probability measures.
To this end, given two probability measures $\pi, \pi' \in \Pcal(\Rbb^d)$, we define the \emph{total variation} (TV) distance as
\begin{align}\label{eq:tv_dist}
        \|\pi' - \pi\|_{\rm TV} \coloneqq \sup_{A\in \Bcal(\Rbb^d)} \left|\int_A \pi'(\d\bfz) - \int_A \pi(\d\bfz)\right|.
\end{align}
We say that the Markov chain $\{\bfZ^t\}_{t\ge 0}$ converges to $\pi$ if $\lim_{t\to\infty}\|\pi^t-\pi\|_{\rm TV}=0$.
We denote convergence in total variation by $\pi^t \xrightarrow{\rm TV} \pi$.
We refer to $\pi$ as the invariant or stationary measure of the Markov transition kernel, $k$, if $\int_{\Rbb^d} k(\bfz,\cdot)\pi(\d\bfz)=\pi(\cdot)$.

We now review mixing times and asymptotic variance to estimate the speed of convergence of Markov chains.
We then describe how spectral properties of the Markov transition kernel connect to both mixing times and asymptotic variance.

\paragraph{Total Variation Mixing Time}
The mixing time is the minimum number of iterations required to satisfy a prescribed TV threshold $\varepsilon > 0$; that is,
    \begin{align}\label{eq:mixing_time}
        t_{\rm mix}(\varepsilon; \pi^0) \coloneqq \min\left\{t \mid \|\pi^t - \pi\|_{\rm TV}\le \varepsilon\right\}
\end{align}
where $\pi$ is the stationary distribution.
Note, we explicitly encode the initial distribution $\pi^0$ in the mixing time notation. 
In practice, quantifying exact mixing times is difficult, thus one estimates them via a computable upper bound, $ t_{\rm mix}^\uparrow(\varepsilon; \pi^0)$.

Intuitively, the closer $\pi^0$ is to the target distribution, $\pi$, the faster an MCMC method will likely mix. To quantify this, we also introduce the notion of a warm start. If $\pi^0$ has density $\pi^0(\bfz)$, we say that it is a warm start for $\pi$ if, for some $\beta \ge 1$,
\begin{align}
\label{eq:characterization_beta_warm}
    \pi^0(\bfz)\le \beta \pi(\bfz),\qquad\text{for almost every }\bfz\in\mathbb{R}^d.
\end{align}
We define $\warmstart{\pi^0}{\pi}$ as the infimum over all $\beta$ for which~\eqref{eq:characterization_beta_warm} holds.
When it is clear in context, for the sake of notation, we will suppress the dependency on $\pi$, i.e., $\warmstart{\pi^0}{\pi} = \warmstart{\pi^0}$.
If $\beta \approx 1$, the initial distribution is close to the target, and thus one expects fast mixing. 
Mixing time upper bounds typically depend on the initial distribution $\pi^0$ only via the warm start constant.

\paragraph{Asymptotic Variance}
A major application of Monte Carlo sampling is to approximate integrals of the form $\int_{\Rbb^d}f(\bfz)\pi(\d\bfz)$ via an
empirical average, $\frac{1}{T}\sum_{t=1}^Tf(\bfZ^t)$, where $\{\bfZ^t\}_{t \ge 0}$ is a Markov chain.
The quality of this estimator can be quantified by the \textit{asymptotic variance}
\begin{align}\label{eq:asymptotic_variance}
    \operatorname{var}(f,k) \coloneqq \lim_{T\to\infty}T\cdot \operatorname{var}\left(\textstyle \frac{1}{T}\sum_{t=1}^Tf(\bfZ^t)\right),
\end{align}
which is well-defined when $f$ is $\pi$-square-integrable.
The asymptotic variance appears in the Central Limit Theorem for Markov chains as
\begin{align*}
    \sqrt{T}\left(\textstyle \frac{1}{T}\sum_{t=1}^Tf(\bfZ^t)-\int_{\Rbb^d}f(\bfz)\pi(d\bfz)\right)\xrightarrow[]{\mathcal{D}}\mathcal{N}(0,\operatorname{var}(f,k)).
\end{align*}
where $\xrightarrow[]{\mathcal{D}}$ means convergence in distribution.
A lower asymptotic variance implies a better quality  Monte Carlo estimator for a fixed number of iterations, $T$.

\paragraph{Spectral Gap}

The spectral properties of the Markov transition kernel connect to both mixing time and asymptotic variance.
Viewing $k$ as an operator on the Hilbert space of $\pi$-square-integrable functions, let $\Scal(k)$ be the spectrum of the kernel.
It can be shown that $\Scal(k) \subset [-1,1]$ if $k$ is $\pi$-reversible\footnote{A Markov transition kernel $k(\bfz,\d\hat\bfz)$ is $\pi$-reversible if it satisfies the detailed balance condition $\pi(\d\bfz)k(\bfz,\d\hat\bfz)=\pi(\d\hat\bfz)k(\hat \bfz,\d\bfz)$ on $\mathbb{R}^d\times \mathbb{R}^d$.}, i.e., self-adjoint in the $\pi$-weighted Hilbert space.
Two quantities of interest are the \emph{right spectral gap}, $\operatorname{Gap}_R(k)$, and the \emph{absolute spectral gap}, $\operatorname{Gap}(k)$, defined, respectively, as
    \begin{align}\label{eq:spectral_gap}
        \operatorname{Gap}_R(k)
        \coloneqq 1 - \sup \{\lambda : \lambda \in \Scal_0(k)\}\quad \text{and} \quad
        \operatorname{Gap}(k) \coloneqq 1 - \sup \{|\lambda| : \lambda \in \Scal_0(k)\},
    \end{align}
where $\Scal_0(k)$ is the spectrum of the kernel restricted to the subspace of mean-zero functions.

Under $\pi$-irreducibility of the chain\footnote{Informally, a chain is $\pi$-irreducible if it is able to explore the entire state space.},  the largest eigenvalue of $k$ is always $1$ with multiplicity $1$.
The second largest eigenvalue quantifies the ``bottlenecks'' in the target distribution, $\pi$, i.e., regions of low probability that separate regions of high probability and are difficult for the chain to cross \cite{lawler_bounds_1988,qin_convergence_2024, milinanni_large_2025}.
A larger right spectral gap indicates that the chain crosses these bottlenecks more easily. This improved mobility across the state space ensures that the chain rapidly forgets the dependence on its initial state, thus reducing mixing time, and decreases the autocorrelation between successive samples, lowering asymptotic variance.

These intuitive connections between spectral gap, mixing times, and asymptotic variance can be made more rigorous. In particular, the reduction in the chain's autocorrelation bounds the asymptotic variance via the right spectral gap as
\begin{align*}
    \operatorname{var}(f,k) \le  \tfrac{2-\operatorname{Gap}_R(k)}{\operatorname{Gap}_R(k)}\operatorname{var}_\pi(f),
\end{align*}
where $\operatorname{var}_\pi(f)\coloneqq \int_{\Rbb^d}f(\bfz)^2\pi(\d\bfz)-\left(\int_{\Rbb^d}f(\bfz)\pi(\d\bfz)\right)^2$ is the variance of $f$ under $\pi$ \cite{douc_markov_2018, andrieu_explicit_2024}.
Moreover, the absolute spectral gap provides a bound to the mixing times.
Following the discussion in~\cite{andrieu_explicit_2024}, if the initial distribution $\pi^0$ is absolutely continuous with respect to $\pi$
and $\chi^2(\pi^0,\pi)$, the $\chi^2$ divergence between $\pi^0$ and $\pi$, is finite, then $t_{\rm mix}(\varepsilon; \pi^0) \leq t_{\rm mix}^\uparrow(\varepsilon; \pi^0)$ where
\begin{align*}
    t_{\rm mix}^\uparrow(\varepsilon;\pi^0) \coloneqq \left\lceil \tfrac{\log\left(2\varepsilon/\sqrt{\chi^2(\pi^0,\pi)}\right)}{\log(1-\text{Gap}(k))} \right\rceil.
\end{align*}
\subsection{Metropolis-Hastings Algorithm}
\label{sec:metropolis_hastings_mixing_times}
The convergence of a Markov chain can be guaranteed if  certain properties of the transition kernel and MCMC algorithm hold.
Metropolis-Hastings (MH) Markov chains \cite{metropolis_equation_1953, hastings_monte_1970} translate these properties into mild conditions on the target density $\pi$ and the proposal distribution $q(\d\bfz^\star\mid\bfz^t)$~\cite[Corollary~7.5 and 7.7]{robert_monte_2004}, which we assume to hold throughout this paper.
Procedurally, given state $\bfz^t\in \Rbb^d$, an MH algorithm proposes a next state, $\bfz^\star$, by sampling from a proposal distribution, $q(\cdot \mid \bfz^t)$. 
The proposal is then accepted with probability
    \begin{align}\label{eq:acceptance_probability}
        \alpha(\bfz^t, {\bfz}^\star) = \min \left\{1, \frac{\tilde{\pi}({\bfz}^\star)q(\bfz^t \mid {\bfz}^\star)}{\tilde{\pi}(\bfz^t)q({\bfz}^\star \mid \bfz^t)  }\right\}
    \end{align}
where $\tilde{\pi}$ is the (unnormalized) target probability density function. 
Upon acceptance, the sample path proceeds by setting $\bfz^{t+1} \gets {\bfz}^\star$.
Otherwise, the proposal is rejected and the path remains at $\bfz^{t+1} \gets \bfz^t$. 
The accept/reject step is the hallmark of MH algorithms. 
The resulting MH Markov chain is characterized by the transition kernel
\begin{align}
    \label{eq:MH_transition_kernel}
    k(\bfz,\d\hat\bfz)=q(\hat\bfz\mid\bfz)\alpha(\bfz,\hat\bfz)+r(\bfz)\delta_{\bfz}(\d\hat\bfz),
\end{align}
where $\alpha$ is the acceptance probability as in~\eqref{eq:acceptance_probability}, and $r$ is the probability of rejecting the proposal and is given by
\begin{align}
\label{eq:rejection_probability}
    r(\bfz)=1-\int_{\bbR^{d}} q(\tilde\bfz\mid\bfz)\alpha(\bfz\mid\tilde\bfz)\d\tilde\bfz.
\end{align}

The proposal mechanism, $q(\cdot \mid \bfz^t)$, gives rise to a wide range of MH algorithms.
For example, the Metropolis-adjusted Langevin algorithm (\MALA) informs proposals using the gradient of the log density function. 
Specifically,  $q_{\MALA}(\cdot \mid \bfz^t) = \Ncal(\bfz^t + h \nabla \log \tilde{\pi}(\bfz^t), 2 h \bfI_d)$ where $\Ncal$ denotes a Gaussian distribution and $h > 0$ is a step size. 
Thus, a \MALA proposal, $\bfz^\star \sim q_{\MALA}(\d\bfz^\star \mid \bfz^t)$, is constructed by
    \begin{align}\label{eq:mala_proposal}
        \bfz^\star = \bfz^t + h \nabla \log \tilde{\pi}(\bfz^t) + \sqrt{2h}\bfzeta^t, \qquad \text{where $\bfzeta^t\sim \Ncal({\bf0}_d, \bfI_d)$.}
    \end{align}
\section{\MarCo: Marginal MCMC with Exact Conditional Draws}
\label{sec:methodology}
We propose a new MCMC algorithm for sampling from $\pjoint(\bfx, \bfy)$ 
that leverages the conditional-marginal factorization \cref{eq:joint_factorization} to improve sampling efficiency. 
Our algorithm combines MH sampling from the \textcolor{blue}{\underline{\texttt{Mar}}}ginal distribution, ${\pi}_{\rm marg}(\d\bfy)$, with sampling from the corresponding exact \textcolor{blue}{\underline{\texttt{Co}}}nditional distribution, $\pcond(\d\bfx \mid \bfy)$; hence, we call our method \textcolor{blue}{\MarCo}.
The \MarCo algorithm is thus designed for sampling problems where
one can \emph{evaluate} the unnormalized marginal density $\tilde\pi_{\rm marg}(\bfy)$ and generate samples from the \emph{exact} conditional distribution; e.g., densities with conditional Gaussian distributions. 
We formalize the \MarCo algorithm and describe the Markov transition kernel of the underlying stochastic process in \Cref{sec:marco_algorithm}. Then, we show that the convergence behavior of \MarCo in the joint space coincides with the MCMC behavior in the marginal space; see \Cref{sec:marco_convergence}. A comparison with existing algorithms with a joint accept/reject strategy is provided in \Cref{sec:marco_comparisons}.

\subsection{The \MarCo Algorithm and Transition Kernel}
\label{sec:marco_algorithm}

The \MarCo algorithm constructs a Markov chain in the joint space $\{(\bfX^t, \bfY^t)\}_{t \ge 0}$ %
in a two-phased approach: first, the $\bfy$-chain is generated targeting the marginal distribution, followed by exact conditional sampling of the $\bfx$-chain, conditioned on the $\bfy$-chain values.

For the $\bfy$-chain, \MarCo uses an MH algorithm that targets the marginal distribution, $\pmarg(\d\bfy)$. In particular, it uses a proposal distribution $q_{\rm marg}(\d\bfy^\star \mid \bfy)$ on $\bbR^\ny$, and performs the MH acceptance/rejection step with the acceptance probability
	\begin{align}\label{eq:marco_acceptance_probability}
	        \alpha_{\rm marg}(\bfy^t, \bfy^\star) = \min \left\{1, 
        \frac{
        \tilde{\pi}_{\rm marg}(\bfy^\star) q_{\rm marg}(\bfy^t \mid \bfy^\star)
        }{
        \tilde{\pi}_{\rm marg}(\bfy^t) q_{\rm marg}(\bfy^\star \mid \bfy^t)
        }
        \right\}.
	\end{align} 
Note that the (unnormalized) marginal density, proposal mechanism, and acceptance probability require only the $\bfy$-coordinate of the joint \MarCo chain. 
During this marginal MH sampling,  \MarCo obtains a sample path, $\{\bfy^t\}_{t=0}^T$. 
To obtain a corresponding realization of the $\bfx$-chain, $\{\bfx^t\}_{t=0}^T$, \MarCo  draws exact samples from the conditional distribution, i.e., $\bfx^t \sim \pcond(\d\bfx \mid \bfy^t)$ for each $t=0,\dots, T$. 
We summarize the \MarCo procedure in \Cref{alg:MarCo}.

\begin{algorithm}[t]
  \caption{\MarCo: Marginal MCMC + Conditional Draws}
  \label{alg:MarCo}
  \begin{algorithmic}[1]
    \Statex {\bf Inputs:} %
        \Statex \begin{tabular}{clcl}
		$\bullet$ & $\pmarg(\bfy)$ & : & evaluable (unnormalized) marginal density \\
		$\bullet$ & $\cdot \sim \pcond(\d\bfx \mid \bfy)$ & : &  algorithm to sample conditional density exactly\\
		$\bullet$ & $q_{\rm marg}(\d\bfy^\star \mid \bfy)$ & : & proposal density in $\bfy$-space\\
		$\bullet$ & $\bfy^0 \in \Rbb^{\ny}$ & : & initial point\\
		$\bullet$ & $T \in \Nbb_{\ge 0}$ & : & maximum number of MCMC steps 
		\end{tabular}
		
    \Statex {\bf Output:} Sample path $\{(\bfx^t,\bfy^t)\}_{t=0}^T$ targeting $\pjoint(\d\bfx,\d\bfy)$
 	\Statex \vspace{-9pt}%
		\Statex \colorbox{lightgray!25}{\makebox[0.985\linewidth][l]{MCMC for $\pmarg(\d\bfy)$}} 
		\vspace{-20pt} 
		\Statex %
    \For{$t=0,\dots,T-1$}
      \State Propose $\bfy^\star \sim q_{\rm marg}(\d \bfy^\star\mid \bfy^t)$
      \State Compute acceptance probability $\alpha_{\rm marg}(\bfy^t, \bfy^\star)$ \Comment{\Cref{eq:marco_acceptance_probability}}
      \State Draw $u^t\sim \operatorname{Unif}(0,1)$ %
      \If{$u^t\leq \alpha_{\rm marg}(\bfy^t,\bfy^\star)$}
        \State Set $\bfy^{t+1}\gets \bfy^\star$ \Comment{accept}
      \Else
        \State Set $\bfy^{t+1}\gets \bfy^t$ \Comment{reject}
      \EndIf
    \EndFor
     		\Statex \vspace{-9pt}  %
		\Statex \colorbox{lightgray!25}{\makebox[0.985\linewidth][l]{Draws from $\pcond(\d\bfx \mid \bfy)$}} 
		\vspace{-20pt} 
		\Statex%
    \For{$t=0,\dots,T$} \Comment{parallelizable}
      \State Draw $\bfx^t \sim \pcond(\d \bfx\mid\bfy^t)$
    \EndFor
  \end{algorithmic}
\end{algorithm}

Despite the two-phased approach, \MarCo does indeed form a Markov chain on the joint space. %
Its transition kernel reflects the conditional-marginal structure over which the algorithm proceeds. 
Specifically, the \MarCo transition kernel is
    \begin{align}\label{eq:marco_transition_kernel}
        k_{\tt M}((\bfx, \bfy), (\d\hat{\bfx}, \d\hat{\bfy})) 
        =\pcond(\d\hat{\bfx} \mid \hat{\bfy}) k_{\rm marg}(\bfy, \d\hat{\bfy}),
    \end{align}
where the marginal kernel, $k_{\rm marg}$, is a standard MH transition kernel on the $\bfy$-space given by
\begin{align}\label{eq:marco_transition_kernel_marginal}
    k_{\rm marg}(\bfy, \d\hat{\bfy}) = q_{\rm marg}(\hat{\bfy} \mid \bfy) \alpha_{\rm marg}(\bfy, \hat{\bfy})\, \d\hat{\bfy}+ r_{\rm marg}(\bfy) \delta_{\bfy}(\d \hat{\bfy})
    \end{align}
where $\delta_{\bfy}(\d\hat{\bfy})$ denotes the Dirac measure centered at $\bfy$ and 
\begin{align}\label{eq:marco_transition_kernel_rejection}
        r_{\rm marg}(\bfy) = 1 - \int_{\Rbb^\ny} q_{\rm marg}(\tilde{\bfy} \mid \bfy) \alpha_{\rm marg}(\bfy, \tilde{\bfy}) \, \d\tilde{\bfy}
    \end{align}
is the probability of rejecting any proposal sampled from $\bfy$.
The hallmark of the \MarCo transition kernel, $k_{\tt M}$, is the multiplicative term $\pcond(\d \hat{\bfx} \mid \hat{\bfy})$. %
This guarantees that \MarCo draws a new sample in the $\bfx$-chain from the exact conditional, regardless of acceptance \emph{or} rejection of ${\bfy}^\star$.
\subsection{Analyzing \MarCo Convergence: The Marginal Is All You Need}
\label{sec:marco_convergence}

By design, the quality of the \MarCo-generated samples in the joint space depends on the performance of the MH algorithm targeting the marginal distribution.
In this section, we prove that the convergence behavior of the \MarCo Markov chain in the joint space is inherited from the convergence behavior of the MH algorithm in the marginal space (see \Cref{thm:MarCo_inherits_marginal}).
We begin with two intermediate observations connecting the marginal and \MarCo Markov chains.
First, we prove that the conditional-marginal factorization of the joint distribution is preserved via the \MarCo transition kernel structure.

\begin{lemma}[\MarCo Preserves Conditional-Marginal Factorization]
\label{lem:marco_preserves_factorization}
Let the initial distribution factorize as $\pi_{\tt M}^0(\d\bfx,\d\bfy) =\pcond(d\bfx\mid\bfy) \mu^0(\bfy)$.
Then, for every $t\ge0$, the law of the $t$-th iterate of \MarCo satisfies
\begin{align}
\label{eq:marco_law_factorization}
    \pi^t_{\tt M}(\d\bfx, \d\bfy)=\pcond(\d\bfx\mid\bfy) \mu^t(\d\bfy).
\end{align}
\end{lemma}

\begin{proof}
We argue by induction on $t$. The claim is true at $t=0$ by assumption. If it holds at time $t$, then by~\eqref{eq:marco_transition_kernel},
\begin{align*}
    \pi^{t+1}_{\tt M}(\d\hat\bfx, \d\hat\bfy)
    &=
    \iint_{\Rbb^\nx \times \Rbb^\ny} k_{\tt M}((\bfx,\bfy),(\d\hat\bfx, \d\hat\bfy)) \, \pi^t_{\tt M}(\d\bfx,\d\bfy) \\
    &=
    \iint_{\Rbb^\nx \times \Rbb^\ny} \pcond(\d\hat\bfx\mid\hat\bfy) k_{\rm marg}(\bfy, \d\hat\bfy)\, \pcond(\d\bfx\mid\bfy) \,\mu^t(\d\bfy)\\
    &=\pcond(\d\hat\bfx\mid\hat\bfy)  \left(\int_{\Rbb^\ny} k_{\rm marg}(\bfy,\d\hat\bfy)\,\mu^t(\d\bfy)\right) \\
    &=\pcond(\d\hat\bfx\mid\hat\bfy) \,\mu^{t+1}(\d\bfy).
\end{align*}
Hence~\eqref{eq:marco_law_factorization} follows for all $t\ge0$.
\end{proof}

The following lemma, adapted from \cite[Proposition 7.2]{polyanskiy_information_2025}, relates the convergence properties of the chain on the marginal space with the \MarCo chain on the joint space.

\begin{lemma}[Total Variation Preservation]
\label{lem:lifted_measure}
The map $\mu(d\bfy)\mapsto \pcond(\d\bfx \mid \bfy)\mu(\d \bfy)$, which lifts marginal distributions on $\Rbb^\ny$ to joint distributions on $\Rbb^{\nx+\ny}$, preserves the TV distance. That is, for any two distributions $\mu_1(d\bfy),\mu_2(d\bfy)$ on $\bbR^\ny$,

	\begin{align}
		\|\mu_1(\d\bfy) - \mu_2(\d\bfy)\|_{\rm TV} = \|\pcond(\d\bfx \mid \bfy)\mu_1(\d \bfy) - \pcond(\d\bfx \mid \bfy)\mu_2(\d \bfy)\|_{\rm TV}.
	\end{align}
\end{lemma}

\begin{proof}
The result follows directly from Property 5 of Proposition 7.2 in \cite{polyanskiy_information_2025}.
\end{proof}

We combine the two lemmas to prove the important property
that a \MarCo Markov chain inherits both the convergence properties and the mixing time of the marginal chain.

\begin{theorem}[\MarCo Convergence]
\label{thm:MarCo_inherits_marginal}
Let $\mu^t(\d\bfy)$ denote the distribution of the $t$-th iterate of an MH Markov chain
with initial distribution $\mu^0(\d\bfy)$.
Then, the following hold.
	\begin{enumerate}[leftmargin=*, label=(\alph*)]
	\item If $\mu^t(\d\bfy) \xrightarrow{\rm TV} \pmarg(\d \bfy)$, then $\pi_{\tt M}^t(\d\bfx, \d\bfy) \xrightarrow{\rm TV} \pjoint(\d\bfx, \d\bfy)$ where $\pi_{\tt M}^t$ is the distribution of the $t$-th iterate of the \MarCo chain as given in \Cref{lem:marco_preserves_factorization}.
	\item The mixing times of the marginal and \MarCo chains coincide; that is,
		\begin{align*}
		t_{{\rm mix}}(\varepsilon; \pi_{\tt M}^0(\d\bfx, \d\bfy)) = t_{{\rm mix}}(\varepsilon; \mu^0(\d\bfy)).
		\end{align*}
	\end{enumerate}

\end{theorem}

\begin{proof}
The proof follows directly from the %
preservation of total variation distance under the lifting introduced in~\Cref{lem:lifted_measure},
and the definition of mixing time in~\eqref{eq:mixing_time}.
\end{proof}
\subsection{Comparisons with Joint Accept/Reject Strategies}
\label{sec:marco_comparisons}

A distinguishing feature of a \MarCo Markov chain is that the accept/reject mechanism is applied only to the marginal chain. 
Traditional MH algorithms that converge to $\pjoint$ apply the accept/reject mechanism in the joint space. 
Such algorithms, denoted by \Joint, have transition kernels of the form
\begin{align}
\label{eq:jointMH_transition_kernel}
    \begin{split}
    k_{\tt J}((\bfx, \bfy), (\hat{\bfx}, \hat{\bfy}))
    &= q_{\rm joint}((\hat{\bfx}, \hat{\bfy}) \mid (\bfx, \bfy))\alpha_{\rm joint}((\bfx, \bfy), (\hat{\bfx}, \hat{\bfy}))\, \d\hat\bfx\, \d\hat\bfy\\
    &\quad + r_{\rm joint}(\bfx, \bfy) \delta_{(\bfx, \bfy)}(\d\hat\bfx, \d\hat{\bfy}),
    \end{split}
\end{align}
which mirrors~\eqref{eq:marco_transition_kernel}, with $\bfz=(\bfx,\bfy)$, $q=q_{\rm joint}$, $\alpha=\alpha_{\rm joint}$ as in~\eqref{eq:acceptance_probability}, and $r=r_{\rm joint}$ as in~\eqref{eq:rejection_probability}.

The \Joint strategy that is most similar to \MarCo is the one-block algorithm. 
Specifically, \OneblockMCMC proposes $(\bfx^\star, \bfy^\star)$ using an MCMC proposal that splits into a proposal on the marginal space, $\bfy^\star \sim q_{\rm marg}(\d\bfy \mid \bfy)$, followed by an exact conditional draw $\bfx^\star \sim \pcond(\d\bfx \mid \bfy^\star)$. The full proposal is hence
\begin{align*}
    q_{\rm joint}((\bfx^\star, \bfy^\star) \mid (\bfx, \bfy)) = \pcond(\d\bfx \mid \bfy^\star)q_{\rm marg}(\d\bfy \mid \bfy).
\end{align*}
Thanks to the use of the exact conditional sampling, the accept/reject functions $\alpha_{\rm joint}$ and $r_{\rm joint}$ simplify in such a way that they lose the dependency on $\bfx$, and their expressions coincide with those of $\alpha_{\rm marg}(\bfy,\hat\bfy)$ in~\eqref{eq:marco_acceptance_probability} and $r_{\rm marg}(\bfy)$ in~\eqref{eq:marco_transition_kernel_rejection}, respectively.
The corresponding \OneblockMCMC transition then simplifies from the general \Joint expression in~\eqref{eq:jointMH_transition_kernel} to
\begin{equation}
\label{eq:oneblock_transition_kernel}
	k_{\tt B}((\bfx, \bfy), (\d\hat\bfx, \d\hat\bfy))
	=  \pcond(\d\hat\bfx \mid \hat\bfy) q_{\rm marg}(\hat{\bfy} \mid \bfy) \alpha_{\rm marg}(\bfy, \hat{\bfy}) \,\d\hat{\bfy}  + r_{\rm marg}(\bfy) \delta_{(\bfx,\bfy)}(\d\hat\bfx, \d \hat{\bfy}).
\end{equation}

Notably, the \OneblockMCMC sampling strategy uses the same proposal mechanism as \MarCo. In fact, upon acceptance, the two strategies have the same behavior: the chain state moves from $(\bfx,\bfy)$ to the proposed (and accepted) $(\bfx^\star,\bfy^\star)$. This is reflected in the fact that
the first term of the \OneblockMCMC transition kernel, $k_{\tt B}$, is identical to the first term of the \MarCo kernel, $k_{\tt M}$, in \eqref{eq:marco_transition_kernel}. 
However, upon rejection, \OneblockMCMC behaves as a traditional \Joint algorithm and rejects the entire joint proposal, $(\bfx^\star, \bfy^\star)$, as indicated by the Dirac measure $\delta_{(\bfx,\bfy)}$ in the second term of $k_{\tt B}$ in \eqref{eq:oneblock_transition_kernel}. %
In contrast, the second term in the \MarCo kernel $k_{\tt M}$ involves the Dirac measure $\delta_{\bfy}$, meaning that only the $\bfy$-component might undergo rejection.
Since the second term in the \OneblockMCMC transition kernel $k_{\tt B}$ differs from the corresponding term in the \MarCo transition kernel $k_{\tt M}$, it is clear that the two strategies implement two different stochastic processes.

In \Cref{sec:theory_marco_vs_one_block}, we show how the subtle, yet important, distinction in rejection behavior enables \MarCo to converge faster than \OneblockMCMC. 
We illustrate the differences between \Joint, \OneblockMCMC, and \MarCo in \Cref{fig:chain_comparison}.

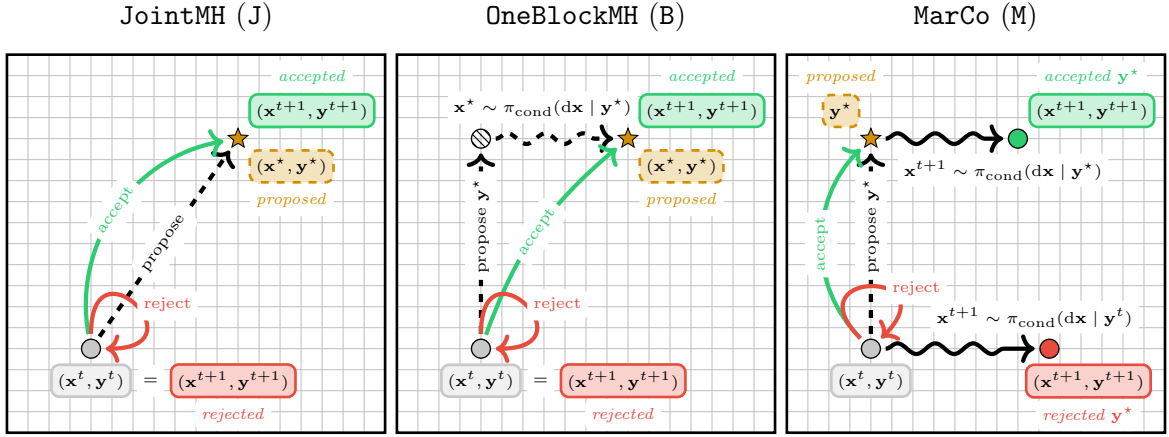
\begin{figure}
   \centering

\definecolor{colorreject}{RGB}{231, 76, 60}
\definecolor{coloraccept}{RGB}{46, 204, 113}
\definecolor{colorstart}{RGB}{200, 200, 200}
\definecolor{colorpropose}{RGB}{210, 142, 0}

\begin{tikzpicture}

\tiny

\def\xymax{9}

\pgfplotsset{m1/.style={mark=*, mark size=3pt, mark options={draw=black, line width=0.5pt}}}
\tikzset{n1/.style={draw=black, line width=0.5pt, circle, minimum width=0.25cm, inner sep=0pt, outer sep=2pt}}
\tikzset{n11/.style={draw=black, line width=0.5pt, rectangle, minimum width=0.25cm, minimum height=0.25cm, inner sep=0pt, outer sep=2pt}}
\tikzset{n2/.style={draw=none, rectangle, minimum height=12pt, fill=white, rounded corners=3pt}} %

\def\xS{-5}
\def\yS{-5}

\def\xJ{2}
\def\yJ{5}

\pgfmathsetmacro{\xM}{\xS}
\pgfmathsetmacro{\yM}{\yJ}

\begin{groupplot}[group style={
		columns=3, rows=1,
		horizontal sep=0.01\linewidth},
		scale only axis,
		width=0.32\linewidth, height=0.32\linewidth,
		axis line style={line width=1pt, -, black},
		grid, grid style={lightgray, line width=0.25pt},
		ticks=none,
		xmin=-\xymax, xmax=\xymax, xtick={-\xymax, ..., \xymax}, xticklabel=\empty,
		ymin=-\xymax, ymax=\xymax,  ytick={-\xymax, ..., \xymax}, yticklabel=\empty
		]

	\nextgroupplot[title={\normalsize\Joint (\Jtt)}]

		\node[n1, fill, colorstart, draw=black] (S) at (axis cs:\xS, \yS) {};
		\node[n2, above=0.0cm of S.south, anchor=north, draw=colorstart, line width=1pt, fill=colorstart!25] (textS) {$(\bfx^t, \bfy^t)$};

		\node[n1, draw=none, fill=none] (P) at (axis cs: \xJ, \yJ) {};
		\node[draw=black, star, star points=5, star point ratio=2.25, fill=colorpropose, scale=0.5]  at (P) {};

		\node[n2, above right=0.0cm of P.north east, anchor=south west, draw=coloraccept, line width=1pt, fill=coloraccept!25, draw] (textP) {$(\widehat{\bfx})$};

		\node[n2, below right=0.0cm of P.south east, anchor=north west, draw=colorpropose, line width=1pt, fill=colorpropose!25, draw, dashed] (textP) {$(\bfx^\star, \bfy^\star)$};
		\node[n2, below=0.0cm of textP.south, anchor=north] {\color{colorpropose} \emph{proposed}};

		\node[n2, above right=0.0cm of P.north east, anchor=south west, draw=coloraccept, line width=1pt, fill=coloraccept!25, draw] (textA) {$(\bfx^{t+1}, \bfy^{t+1})$};
		\node[n2, above=0.0cm of textA.north, anchor=south] {\color{coloraccept} \emph{accepted}};

		\node[n2, right=0.5cm of textS.east, anchor=west, draw=colorreject, line width=1pt, fill=colorreject!25] (textR) {$(\bfx^{t+1}, \bfy^{t+1})$};
		\node[n2, below=0.0cm of textR.south, anchor=north] {\color{colorreject} \emph{rejected}};
		\node[n2] at ($(textS.east)!0.5!(textR.west)$) {$=$};

		\draw[ultra thick, ->, black, dashed] (S) -- node[n2, pos=0.5, sloped] {propose} (P);
		\draw[ultra thick, ->, coloraccept] (S)  to[bend left=45] node[n2, midway, sloped] {accept} (P);
		\draw[ultra thick, ->, colorreject] (S)  edge [loop right, looseness=12, min distance=0.5cm, out=90, in=0] node[n2, midway] {reject} (S);

	\nextgroupplot[title={\normalsize \OneblockMCMC (\Btt)}]

		\node[n1, fill, colorstart, draw=black] (S) at (axis cs:\xS, \yS) {};
		\node[n2, above=0.0cm of S.south, anchor=north, draw=colorstart, line width=1pt, fill=colorstart!25] (textS) {$(\bfx^t, \bfy^t)$};

		\node[n1, draw=black, pattern=north west lines] (H) at (\xM, \yM) {};

		\node[n1, draw=none, fill=none] (P) at (axis cs: \xJ, \yJ) {};
		\node[draw=black, star, star points=5, star point ratio=2.25, fill=colorpropose, scale=0.5]  at (P) {};

		\node[n2, below right=0.0cm of P.south east, anchor=north west, draw=colorpropose, dashed, line width=1pt, fill=colorpropose!25, draw] (textP) {$(\bfx^\star, \bfy^\star)$};
		\node[n2, below=0.0cm of textP.south, anchor=north] {\color{colorpropose} \emph{proposed}};

		\node[n2, above right=0.0cm of P.north east, anchor=south west, draw=coloraccept, line width=1pt, fill=coloraccept!25, draw] (textA) {$(\bfx^{t+1}, \bfy^{t+1})$};
		\node[n2, above=0.0cm of textA.north, anchor=south] {\color{coloraccept} \emph{accepted}};

		\node[n2, right=0.5cm of textS.east, anchor=west, draw=colorreject, line width=1pt, fill=colorreject!25] (textR) {$(\bfx^{t+1}, \bfy^{t+1})$};
		\node[n2, below=0.0cm of textR.south, anchor=north] {\color{colorreject} \emph{rejected}};
		\node[n2] at ($(textS.east)!0.5!(textR.west)$) {$=$};

		\draw[ultra thick, ->, black, dashed] (S) -- node[n2, pos=0.6, sloped] {propose $\bfy^\star$} (H);
		\draw[ultra thick, ->, black, dashed,
			decorate,
    			decoration={snake, amplitude=2pt, segment length=0.5cm, post length=5pt}
   	 		] (H) -- node[n2, pos=0.4, sloped, above=0.175cm, anchor=south]  {$\bfx^\star \sim \pcond( \d \bfx\mid \bfy^\star)$} (P);
		\draw[ultra thick, ->, coloraccept] (S)  to[bend left=15] node[n2, midway, sloped] {accept} (P);
		\draw[ultra thick, ->, colorreject] (S)  edge [loop right, looseness=12, min distance=0.5cm, out=90, in=0] node[n2, midway] {reject} (S);

	\nextgroupplot[title={\normalsize\MarCo (\Mtt)}]

		\node[n1, fill, colorstart, draw=black] (S) at (axis cs:\xS, \yS) {};
		\node[n2, above=0.0cm of S.south, anchor=north, draw=colorstart, line width=1pt, fill=colorstart!25] (textS) {$(\bfx^t, \bfy^t)$};

		\node[n1, fill=none, draw=none] (P) at (axis cs: \xM, \yM) {};
		\node[draw=black, star, star points=5, star point ratio=2.25, fill=colorpropose, scale=0.5]  at (P) {};

		\node[n1, fill, coloraccept, draw=black] (A) at (axis cs: \xJ, \yJ) {};
		\node[n1, fill, colorreject, draw=black] (R) at (axis cs: \xJ+1.5, \yS) {};

		\node[n2, above left=0.0cm of P.north west, anchor=south east, draw=colorpropose, line width=1pt, fill=colorpropose!25, draw, dashed] (textP) {$\bfy^\star$};
		\node[n2, above=0.0cm of textP.north, anchor=south, xshift=0.0cm] {\color{colorpropose} \emph{proposed}};

		\node[n2, above right=0.0cm of A.north east, anchor=south west, draw=coloraccept, line width=1pt, fill=coloraccept!25, draw] (textA) {$(\bfx^{t+1}, \bfy^{t+1})$};
		\node[n2, above=0.0cm of textA.north, anchor=south] {\color{coloraccept} \emph{accepted $\bfy^\star$}};

		\node[n2, below=0.0cm of R.south, anchor=north west, draw=colorreject, line width=1pt, fill=colorreject!25, xshift=-0.325cm] (textR) {$(\bfx^{t+1}, \bfy^{t+1})$};
		\node[n2, below=0.0cm of textR.south, xshift=0.0cm, anchor=north] {\color{colorreject} \emph{rejected $\bfy^\star$}};

		\draw[ultra thick, ->, black, dashed] (S) -- node[n2, pos=0.6, sloped] {propose $\bfy^\star$} (P);
		\draw[ultra thick, ->, coloraccept] (S)  to[bend left=45] node[n2, midway, sloped] {accept} (P);
		\draw[ultra thick, ->, colorreject] (S)  edge [loop right, looseness=12, min distance=0.5cm, out=135, in=45] node[n2, pos=0.6] {reject} (S);

		\draw[ultra thick, ->, black,
		decorate,
    		decoration={snake, amplitude=2pt, segment length=0.5cm, post length=5pt}
   	 	] (P) -- node[n2, pos=1, sloped, below=0.175cm, anchor=north] (sampA) {$\bfx^{t+1} \sim \pcond( \d \bfx\mid \bfy^\star)$} (A);

		\draw[ultra thick, ->, black,
		decorate,
    		decoration={snake, amplitude=2pt, segment length=0.5cm, post length=5pt}
   	 	] (S) -- node[n2, pos=1, above=0.175cm, anchor=south] {$\bfx^{t+1} \sim \pcond( \d \bfx\mid \bfy^t)$} (R);

\end{groupplot}

\end{tikzpicture}

\caption{Left-to-right: illustrative comparison of one step of \Joint, \OneblockMCMC, and \MarCo.  
Every algorithm starts at the same current state, $(\bfx^t, \bfy^t)$, in the lower left corner ({\color{gray} \bf gray} circle). 
The proposed state, $(\bfx^\star, \bfy^\star)$ or $\bfy^\star$, is indicated by a {\color{colorpropose} \bf yellow} star. 
The  proposal dynamics are indicated by dashed arrows ($\dashrightarrow$). 
Solid arrows ($\rightarrow$) designate the actual step to the next state, $(\bfx^{t+1}, \bfy^{t+1})$, after the accept/reject decision has been made. 
We explicitly indicate draws from the conditional distribution, $\pcond$, by squiggly arrows ($\leadsto$). 
Upon acceptance, the chain steps to the next state highlighted by the {\color{coloraccept}\bf green} box (upper right corner). 
Upon rejection, the chain steps to the next state highlighted by the {\color{colorreject}\bf red} box.  
The distinguishing feature of \MarCo is its rejection behavior. 
While the other algorithms remain at the current position upon rejection, \MarCo continues in the $\bfx$-direction by sampling from $\pcond$.
}
\label{fig:chain_comparison}
\end{figure}

\section{Conditionals Improve Conditioning and Mixing Time Upper Bounds}
\label{sec:theory_marco_vs_joint}

In \Cref{thm:MarCo_inherits_marginal} above, we proved that \MarCo acquires the same convergence behavior in the joint space as the underlying MH algorithm for the marginal distribution.
This property of \MarCo can lead to 
significantly improved convergence statistics compared with \Joint.
In this section, we make this argument precise by proving that \MarCo improves, relative to \Joint, the non-asymptotic mixing time bounds.

In particular, our framework naturally applies to mixing time upper bounds such as those developed in~\cite{dwivedi_log-concave_2019,wu_minimax_2022,altschuler_resolving_2022,altschuler_faster_2024,chewi_analysis_2025} for log-concave distributions.
These mixing times bounds depend on the regularity and dimension of the underlying distribution (described in \Cref{subsec:regularity}).
Since \MarCo draws samples from the marginal, one obtains a clear and immediate advantage in terms of dimension over MCMC in the joint space ($\ny$ vs. $\nx + \ny$). Moreover, we show in \Cref{sec:marco_vs_joint_conditioning} that the marginal has nicer regularity properties (in terms of the condition number). Thus, under appropriate tuning of MCMC hyperparameters, we show in \Cref{sec:spectral_consequences} that the mixing time upper bound for \MarCo is smaller than \Joint.
\subsection{Regularity of the Joint Distribution}
\label{subsec:regularity}

Our results rely on global curvature properties of the joint density.
Specifically, we build on the wealth of analysis available for \emph{log-concave} distributions, i.e., those of the form $\pi(\bfz) = \exp(-V(\bfz)),$ where $V$ is convex.
For such densities, we define a few quantities of regularity that provide informative scalar summaries.
We first define tight versions of these curvature constants and then state our regularity assumptions.%

\begin{definition}[Smoothness, Convexity, and Condition number]
  \label{def:tight-condition-number}
  Let $\pi(\bfz) = \exp(-V(\bfz))$ be a probability density on $\bbR^d$, where $V \in C^2(\bbR^d)$ is convex.
  Define its smoothness and convexity constants by
  \begin{align}
    L(\pi) & \coloneqq \inf\left\{ L \geq 0: \nabla^2 V(\bfz) \preceq L\bfI_d \text{ for every }\bfz \in \bbR^d \right\}, \label{eq:tight-smoothness-constant} \\
    m(\pi) & \coloneqq \sup\left\{ m \geq 0: m\bfI_d \preceq \nabla^2 V(\bfz) \text{ for every }\bfz \in \bbR^d \right\}. \label{eq:tight-strong-convexity-constant}
  \end{align}
  If $m(\pi) > 0$, we say that $\pi$ is strongly log-concave, and we define its condition number as
  \begin{equation}
    \label{eq:tight-condition-number}
    \kappa(\pi) \coloneqq L(\pi) / m(\pi).
  \end{equation}
If $m(\pi)=0$, we say that $\pi$ is weakly log-concave.
\end{definition}
The condition number defined in \eqref{eq:tight-condition-number} plays a central role in the upcoming mixing time upper bounds.
We also consider a more general setting based on the \emph{isoperimetric constant} that allows us to get a sharper quantity for mixing \cite{wu_minimax_2022,goodman_geometric_2005}.
\begin{definition}[Isoperimetric Constant and Condition Number]
  \label{def:isoperimetric-constant}
  Let $\pi$ be a probability distribution on $\bbR^d$.
  For nonempty sets $S_1,S_2 \subseteq \bbR^d$, define
  \begin{equation}
    D(S_1,S_2) \coloneqq \inf_{\bfz_1 \in S_1,\,\bfz_2 \in S_2} \|\bfz_1-\bfz_2\|_2.
  \end{equation}
  We say that $\pi$ satisfies an isoperimetric inequality with constant $\psi > 0$ if, for every measurable partition $(S_1,S_2,S_3)$ of $\bbR^d$ with $S_1$ and $S_2$ nonempty,
  \begin{equation}
    \label{eq:isoperimetric-inequality}
    \pi(S_3) \geq \psi\,D(S_1,S_2)\,\pi(S_1)\pi(S_2).
  \end{equation}
  The tightest isoperimetric constant $\psi(\pi)$ is the supremum of all constants $\psi$ for which~\eqref{eq:isoperimetric-inequality} holds.
  Moreover, if $\psi(\pi) > 0$, then define the isoperimetric condition number
  \begin{equation}
    \label{eq:isoperimetric-condition-number}
    \isocond(\pi) \coloneqq L(\pi) / \psi(\pi)^2.
  \end{equation}
\end{definition}
The smoothness constant $L$ controls the largest local curvature of the negative log-density and, consequently, the step size that a gradient-based sampler can safely take.
By contrast, $m$ and $\psi$ quantify aspects of the target geometry that govern the global ease of exploration: $m$ through a uniform lower bound on curvature, and $\psi$ more directly by guaranteeing the absence of bottlenecks. 
Small values of either can correspond to slow movement across the state space\footnote{Both yield conductance bounds, a central mechanism in mixing time theory; see, e.g.,~\cite[Section 3.1.1]{wu_minimax_2022}.}.
The ratios \eqref{eq:tight-condition-number} and \eqref{eq:isoperimetric-condition-number} combine these local and global effects into dimensionless scalars that describe sampling difficulty, and hence impact mixing time bounds.

\begin{remark}
  For a strongly log-concave $\pi$, we have $\psi(\pi)^2 \geq \log(2)^2 m(\pi)$ \cite[Theorem 4.4]{cousins_cubic_2014}.
  Therefore, any strongly log-concave distribution satisfies an isoperimetric inequality and $\isocond(\pi) \leq \log(2)^{-2} \kappa(\pi)$.
  The opposite is not true.
  For example, logistic product measures are log-concave with $m(\pi)=0$ and $L(\pi)<\infty$, while satisfying $\psi(\pi) > 0$ \cite[Section~4.2 and Theorem~4.5]{barthe_isoperimetry_2013}.
  Consequently, $\isocond(\pi)$ remains finite whereas $\kappa(\pi)$ is not defined.
\end{remark}

Next, we state the regularity assumptions for the joint distribution.
\begin{assumption}[Joint Regularity]
  \label{assump:joint_density_regularity}
  Assume the following regularity conditions:
  \begin{enumerate}[leftmargin=*, label={\bfseries (A\arabic*)}]
    \item %
    $\pjoint$ has a twice continuously differentiable negative log-density $V_{\rm joint}$, i.e.,\\
    $V_{\rm joint}(\bfx,\bfy)\coloneqq -\log\pjoint(\bfx,\bfy)\in C^2(\Rbb^\nx \times \Rbb^\ny)$;
\label{assump:continuously_differentiable}
    \item $\pjoint$ has finite smoothness constant, i.e., $L(\pjoint) < \infty$ (see \eqref{eq:tight-smoothness-constant});\label{assump:L_smooth}
    \item $\pjoint$ is log-concave, i.e., $m(\pjoint) \geq 0$ (see \eqref{eq:tight-strong-convexity-constant});\label{assump:log_concave}
    \item $\pjoint$ satisfies an isoperimetric inequality, i.e., $\psi(\pjoint) > 0$ (see \Cref{def:isoperimetric-constant}).\label{assump:isoperimetric}
  \end{enumerate}
\end{assumption}

\subsection{Marginalization Improves Conditioning}
\label{sec:marco_vs_joint_conditioning}

The goal of this section is to show that the isoperimetric condition number  for the marginal, $\isocond(\pmarg)$, is no greater than the isoperimetric condition number for the joint, $\isocond(\pjoint)$. 
To this end, we compare the numerator and denominator of~\eqref{eq:isoperimetric-condition-number} separately. 
First, we relate the isoperimetric constants \eqref{def:isoperimetric-constant} (in the denominator of $\isocond$) in \Cref{prop:marginal-isoperimetry} through direct evaluation of $\pmarg(\bfy)$ and $\pjoint(\bfx, \bfy)$. 
Then, we relate the smoothness constants \eqref{eq:tight-smoothness-constant}  (the numerator of $\isocond$), through linear algebraic spectral arguments. 

\begin{proposition}[Marginalization Isoperimetric Bounds]
  \label{prop:marginal-isoperimetry}
  If $\psi(\pjoint) < \infty$, then
  \begin{equation}
    \label{eq:marginal-isoperimetry}
    \psi(\pjoint) \leq \psi(\pmarg).
  \end{equation}
\end{proposition}

\begin{proof}
  Let $(S_1,S_2,S_3)$ be a measurable partition of $\bbR^\ny$, and define the corresponding partition of the joint space by $A_i \coloneqq \bbR^\nx \times S_i$, for $i \in \{1,2,3\}$.
  By the definition of the marginal, $\pjoint(A_i)=\pmarg(S_i)$ for each $i$.
  Moreover,
  \begin{equation*}
    D(A_1,A_2)
    = \inf_{\substack{\bfx_1,\bfx_2\in\bbR^\nx\\ \bfy_1\in S_1,\,\bfy_2\in S_2}}
    \left( \|\bfx_1-\bfx_2\|_2^2 + \|\bfy_1-\bfy_2\|_2^2 \right)^{1 / 2}
    = D(S_1,S_2),
  \end{equation*}
  where %
  the second equality is obtained by taking $\bfx_1=\bfx_2$.
  Thus, if $\psi$ is any admissible isoperimetric constant for $\pjoint$, applying~\eqref{eq:isoperimetric-inequality} to $(A_1,A_2,A_3)$ gives
  \begin{equation*}
    \pmarg(S_3)  = \pjoint(A_3)
    \geq \psi\,D(A_1,A_2)\,\pjoint(A_1)\pjoint(A_2)
    = \psi\,D(S_1,S_2)\,\pmarg(S_1)\pmarg(S_2).
  \end{equation*}
  Hence, $\psi$ is also an admissible isoperimetric constant for $\pmarg$, and
  \eqref{eq:marginal-isoperimetry} follows by taking the supremum in $\psi$.
\end{proof}

Next, we analyze the spectral properties of the marginal and joint densities to compare the smoothness parameters. 
We start with \Cref{lem:marginal-hessian-identity} to connect the Hessian of the marginal negative log-density (or marginal potential), $V_{\rm marg}: \Rbb^\ny \to \Rbb$ given by
    \begin{align}
     \label{def:Jmarg}
        V_{\rm marg}(\bfy) \coloneqq -\log \pmarg(\bfy)
    \end{align}
to the Hessian of the joint potential, $V_{\rm joint}(\bfx, \bfy)$, with respect to $\bfy$.

\begin{lemma}[Marginal Hessian]
  \label{lem:marginal-hessian-identity}
  Under \Cref{assump:joint_density_regularity}, the marginal potential satisfies,
  \begin{equation}
    \label{eq:marginal-hessian-identity}
    \nabla^2 V_{\rm marg}(\bfy)
    = \mathbb{E}_{\bfx \sim \pcond(\cdot \mid\bfy)}\!\bigl[\nabla_{\bfy}^2 V_{\rm joint}(\bfx,\bfy)\bigr]
    - \operatorname{Cov}_{\bfx\sim \pcond(\cdot \mid\bfy)}\!\bigl[\nabla_{\bfy} V_{\rm joint}(\bfx,\bfy)\bigr].
  \end{equation}
\end{lemma}
\begin{proof}
    This can be viewed as the Fisher--Louis identity \cite[Proposition~10.1.6]{cappe_inference_2005}, specialized to our setting, notation, and assumptions; see Appendix~\ref{appendix:marginal-hessian-identity} for the proof.
\end{proof}

From the connection between the marginal and joint potentials,  we show that the marginal distribution inherits the regularity properties of the joint density regularity assumptions.

\begin{proposition}[Marginal Spectral Bounds]
  \label{prop:marginal-hessian}
  Suppose $\pjoint$ satisfies $\Cref{assump:joint_density_regularity}$.
  Then, for all $\bfy\in \Rbb^\ny$, 
  \begin{align}\label{eq:implication_L}
      \nabla^2 V_{\rm marg}(\bfy) \preceq L(\pjoint) \bfI_{\ny}, \quad \text{which implies} \quad  L(\pmarg) \leq L(\pjoint).
  \end{align}
  Moreover, if $\pjoint$ is strongly log-concave, i.e., $m(\pjoint) > 0$, then, for all $\bfy\in \Rbb^\ny$, 
  \begin{align}\label{eq:implication_m}
      m(\pjoint) \bfI_{\ny} \preceq \nabla^2 V_{\rm marg}(\bfy), \quad \text{which implies} \quad m(\pjoint) \leq m(\pmarg).
  \end{align}
\end{proposition}

\begin{proof}
For the sake of notation, we write $V_{\rm joint}(\bfx,\bfy)$ simply as $V(\bfx, \bfy)$.

First, note the marginal of a log-concave distribution remains log-concave~\cite{prekopa_logarithmic_1973}.
    Now, since $\nabla_{\bfy}^2 V(\bfx, \bfy)$ is a principal submatrix of the full Hessian, the $L$-smoothness of the $\pjoint$ \ref{assump:L_smooth} further implies $\nabla_{\bfy}^2 V(\bfx, \bfy) \preceq L(\pjoint) \bfI_{\ny}$ for all  $(\bfx, \bfy)\in \Rbb^{\nx} \times \Rbb^\ny$.
    Following \Cref{lem:marginal-hessian-identity} and by the positive semidefiniteness of covariance matrices, we obtain the bound,
    \[
        \nabla^2 V_{\rm marg}(\bfy)
        \preceq
        \mathbb{E}_{\bfx\sim \pcond(\cdot \mid\bfy)}\!\bigl[\nabla_{\bfy}^2 V(\bfx,\bfy)\bigr]
        \preceq
        L(\pjoint) \bfI_{\ny}.
    \]
    For the lower bound, the Brascamp--Lieb covariance inequality~\cite[Theorem 4.2]{brascamp_extensions_1976} applied to $\pcond(\bfx\mid\bfy)\propto \exp(-V(\bfx,\bfy))$ gives
    \[
        \operatorname{Cov}_{\bfx \sim \pcond(\cdot\mid\bfy)}\!\bigl(\nabla_\bfy V(\bfx, \bfy)\bigr)
        \preceq
        \mathbb{E}_{\bfx\sim \pcond(\cdot \mid\bfy)}\!\bigl[\nabla_{\bfy\bfx}^2 V(\bfx, \bfy) \left(\nabla_{\bfx}^2 V(\bfx, \bfy)\right)^{-1}\nabla_{\bfx\bfy}^2V(\bfx, \bfy)\bigr],
    \]
    for all $\bfy\in \Rbb^{\ny}$.
    Combined with  \Cref{lem:marginal-hessian-identity}, we obtain
    \begin{equation*}
        \label{eq:marginal-hessian-lower}
        \mathbb{E}_{\bfx\sim \pcond(\cdot \mid\bfy)}\!\bigl[ \underbrace{\nabla_\bfy^2 V(\bfx, \bfy) - \nabla_{\bfy\bfx}^2 V(\bfx, \bfy) \left(\nabla_{\bfx}^2 V(\bfx, \bfy)\right)^{-1}\nabla_{\bfx\bfy}^2V(\bfx, \bfy)}_{\bfS(\bfx, \bfy)} \bigr]
        \preceq  \nabla^2 V_{\rm marg}(\bfy)
    \end{equation*}
    where $\bfS(\bfx, \bfy)$ is the Schur complement of $\nabla^2 V(\bfx, \bfy)$.
    The spectrum of $\bfS(\bfx, \bfy)$ is contained within the spectrum of $\nabla^2 V(\bfx, \bfy)$ \cite[Theorem 5]{smith_interlacing_1992}.
    Following Assumption~\ref{assump:log_concave}, we then get the lower bound $m(\pjoint) \bfI_{\ny} \preceq \nabla^2 V_{\rm marg}(\bfy)$.

    The final implications of curvature constants in \eqref{eq:implication_L} and \eqref{eq:implication_m} follow since $L(\pmarg)$ and $m(\pmarg)$ are defined as the tightest possible constants bounding the spectrum.
\end{proof}
    The proof in \Cref{prop:marginal-hessian} illustrates the connection between the conditional-marginal structure and the underlying spectral properties of the density.
    This mirrors the conditioning theorems of variable projection \cite{chowdhary_boost_2026,newman_train_2021}, which were a major motivation for this work.
    As an aside, \Cref{prop:marginal-hessian} could also be shown from Pr\`{e}kopa's theorem~\cite[Theorem 6]{prekopa_logarithmic_1973}.
    One would consider marginals of constructions of the form $\pjoint(\bfx,\bfy)\exp\left( m(\pjoint) \|\bfy\|_2^2\right)$ or $\pjoint(\bfx,\bfy)\exp\left( -L(\pjoint) \|\bfy\|_2^2\right)$ and use Pr\`{e}kopa to reason about their log-concavity.

Employing both \Cref{prop:marginal-hessian,prop:marginal-isoperimetry}, we conclude with the desired conditioning result for the marginal, and therefore for \MarCo.

\begin{corollary}[Marginalization Improves Conditioning]
  \label{cor:marginal-isocond}
  If $\pjoint$ satisfies~\Cref{assump:joint_density_regularity},
  \begin{equation}
    \label{eq:marginal-isoperimetric-condition-number}
    \isocond(\pmarg) \leq \isocond(\pjoint).
  \end{equation}
In the special case where $\pjoint$ is strongly log-concave, i.e., $m(\pjoint) > 0$, we have
  \begin{equation}
  \label{eq:ineq_kappa_pi}
    \kappa(\pmarg) \leq \kappa(\pjoint).
  \end{equation}
\end{corollary}

\subsection{\MarCo Improves Mixing Time Upper Bounds}
\label{sec:spectral_consequences}
Having proved that marginalization preserves, if not improves, the conditioning of the density, we now turn to the implications of this fact on mixing time upper bounds.
In order to compare the two, we begin with a set of assumptions regarding said bounds and the MCMC strategies being compared.

\begin{assumption} \label{assump:mixing_time_comparison}
To compare mixing time bounds, consider implementations satisfying the following.
    \begin{enumerate}[label={\bfseries (B\arabic*)}, leftmargin=*]
        \item \label{assump:same_algorithm} The transition kernel of \Joint and that of the $\bfy$-marginal chain in \MarCo are instances of the same MCMC kernel family, targeting $\pjoint$ and $\pmarg$, respectively.
        \item \label{assump:initial_distribution} The initial distribution of \Joint is $\pjoint^0(\d\bfx, \d\bfy) = \pcond(\d\bfx \mid \bfy) \pmarg^0(\d\bfy)$,
            and $\pmarg^0(\d\bfy)$ is used as the initial distribution for the $\bfy$-chain in \MarCo.
            Moreover, $\pjoint^0$ is a warm start~\eqref{eq:characterization_beta_warm} for $\pjoint$.
    \item \label{assump:hyperpar_and_monotone_mixing_time} For the MCMC family considered, there exists a hyperparameter tuning function that depends on the mixing accuracy $\varepsilon$ \eqref{eq:mixing_time}, the state space dimension, the warm start parameter $\beta$ \eqref{eq:characterization_beta_warm}, and geometric properties of the target distribution, $\pi\in \Pcal(\Rbb^d)$, from \Cref{assump:joint_density_regularity}; i.e., there exists a map
    \begin{align*}\label{eq:hyperparameter_tuning}
	(\varepsilon,d,\warmstart{\pi^0}, L(\pi),\psi(\pi)) \mapsto \{\text{tuned MCMC hyperparameters}\}.
	\end{align*}
    This tuning function is such that the considered MCMC method with tuned hyperparameters has a mixing time upper bound $t_{\rm mix}^\uparrow(\varepsilon; d,\warmstart{\pi^0},\isocond(\pi))$ that is strictly increasing in $d$ and non-decreasing in $\warmstart{\pi^0}$ and $\isocond(\pi)$.
    \end{enumerate}
\end{assumption}

An example of an MCMC family that satisfies~\ref{assump:hyperpar_and_monotone_mixing_time}
is a minor variation %
of \MALA \cite[Appendix D.1]{altschuler_faster_2024}.
Translating \cite[Theorem~3]{wu_minimax_2022} to our notation, the \MALA hyperparameter, i.e., its step size, is tuned by
\begin{equation*}
  \label{eq:weak-mala-example-step-size}
  h_{\MALA} = \frac{c_0}{L(\pi) \sqrt{d}\,\log^2\left( \max\left\{d,\frac{L(\pi)}{\psi(\pi)^2},\frac{\warmstart{\pi^0}}{\varepsilon},c_2\right\}\right)}
\end{equation*}
and the mixing time upper bound is given by
\begin{equation*}
  \label{eq:weak-mala-example-mixing-bound}
  t_{\rm mix,\, \MALA}^\uparrow(\varepsilon; d,\warmstart{\pi^0},\isocond(\pi))
  = c_1\left\{
      \isocond(\pi)\sqrt{d}\,\log^3\left( \max\left\{d, \isocond(\pi),\frac{\warmstart{\pi^0}}{\varepsilon},c_2\right\}\right)
      + \log\left(\frac{2\warmstart{\pi^0}}{\varepsilon}\right)
    \right\},
\end{equation*}
where $c_0, c_1, c_2 > 0$ are universal constants.
The tuned step size relies only on the parameters $(\varepsilon,d,\warmstart{\pi^0},L(\pi),\psi(\pi))$ and the mixing time upper bound is strictly increasing in $d$ and non-decreasing in $\warmstart{\pi^0}$ and $\isocond$, thereby satisfying Assumption~\ref{assump:hyperpar_and_monotone_mixing_time}.%

First, we prove that the warm start parameter does not increase under marginalization.
\begin{lemma}[Warm Start Under Marginalization]
\label{lem:warm_start_preserved}
	Assume $\pjoint^0$ satisfies \ref{assump:initial_distribution} and is a $\beta$-warm start for the target distribution, $\pjoint$.
	Then, $\pmarg^0$ from \ref{assump:initial_distribution} is also a $\beta$-warm start for the marginal target distribution, $\pmarg$.
    Moreover, $\warmstart{\pmarg^0} \leq \warmstart{\pjoint^0}$.
\end{lemma}

\begin{proof}%
Integrating the warm-start inequality \eqref{eq:characterization_beta_warm} over $\bfx$ gives, for a.e. $\bfy \in \bbR^\ny$,
\[
    \pmarg^0(\bfy)
    = \int_{\bbR^\nx} \pi^0(\bfx,\bfy)\,\d\bfx
    \leq \warmstart{\pjoint^0} \int_{\bbR^\nx} \pi(\bfx,\bfy)\,\d\bfx
    = \warmstart{\pjoint^0} \pmarg(\bfy)
\]
Hence, $\pmarg^0$ is a $\beta$-warm start to $\pmarg$ with constant $\warmstart{\pjoint^0}$.
As $\warmstart{\pmarg^0}$ is defined as the infimum over such constants, we conclude that $\warmstart{\pmarg^0} \leq \warmstart{\pjoint^0}$.
\end{proof}

We now combine the relation between the $\beta$-warm start of $\pjoint^0$ and $\pmarg^0$ in \Cref{lem:warm_start_preserved}, the inequality of isoperimetric condition numbers in \Cref{cor:marginal-isocond}, and the equality of joint and marginal mixing times in \Cref{thm:MarCo_inherits_marginal} to prove our main theorem: that \MarCo achieves a smaller mixing time upper bound than \Joint.

\begin{theorem} \label{thm:mixing_M_le_J}
Suppose \Joint and \MarCo satisfy \Cref{assump:mixing_time_comparison}.
  Then, the mixing time upper bound of \MarCo is strictly smaller than the mixing time upper bound of \Joint.
  That is, for threshold $\varepsilon > 0$,
  \begin{equation}
    \label{eq:mixing_time_M_le_J_weak}
    t_{\rm mix, \MarCo}^{\uparrow}(\varepsilon)
    < t_{\rm mix, \Joint}^{\uparrow}(\varepsilon).
  \end{equation}

\end{theorem}

\begin{proof}
By \ref{assump:hyperpar_and_monotone_mixing_time}, %
$(\varepsilon,\nx + \ny,\warmstart{\pjoint^0}, L(\pjoint),\psi(\pjoint))$ uniquely determine the hyperparameters of the \Joint kernel, whereas $(\varepsilon,\ny,\warmstart{\pmarg^0}, L(\pmarg),\psi(\pmarg))$ uniquely determine those of the kernel for the $\bfy$-chain in \MarCo.
The same assumption guarantees that the mixing time upper bounds of \Joint and \MarCo are
  \begin{align*}
    t_{\rm mix,\Joint}^\uparrow(\varepsilon)
    &= t_{\rm mix}^\uparrow
      \left(\varepsilon; \nx+\ny,\warmstart{\pjoint^0},\isocond(\pjoint)\right) \qquad \text{and}\\
     t_{\rm mix,\,\MarCo}^\uparrow(\varepsilon)
    &= t_{\rm mix}^\uparrow
      \left(\varepsilon; \ny,\warmstart{\pmarg^0},\isocond(\pmarg)\right).
  \end{align*}
  The second equality follows from \Cref{thm:MarCo_inherits_marginal}, which shows that the full \MarCo chain inherits the mixing time upper bound for the marginal chain.
  Recall, marginalization reduces the dimensionality (i.e., $\ny < \nx+\ny$), $\isocond(\pmarg) \leq \isocond(\pjoint)$ by~\eqref{eq:marginal-isoperimetric-condition-number} in~\Cref{cor:marginal-isocond}, and $\warmstart{\pmarg^0} \leq \warmstart{\pjoint^0}$ by~\Cref{lem:warm_start_preserved}.
  Because, by assumption, $t_{\rm mix}^\uparrow$ is strictly increasing in dimension and nondecreasing in isoperimetric condition number and warm-start parameter, the \MarCo mixing time bound is strictly less than the \Joint bound.
\end{proof}

The mixing time upper bound in \Cref{assump:mixing_time_comparison} is stated in terms of the isoperimetric condition number $\isocond$, but if $\pjoint$ is strongly log-concave (hence, $\kappa$ is finite) the above results can be stated in terms of $\kappa$, using inequality~\eqref{eq:ineq_kappa_pi} instead of~\eqref{eq:marginal-isoperimetric-condition-number} in the above proof.
In fact, it is common in the log-concave sampling literature to express mixing time upper bounds in terms of $\kappa$ in the strongly log-concave cases; see, e.g., \cite[Theorem 1]{dwivedi_log-concave_2019} and \cite[Theorem 1]{wu_minimax_2022}.
A treatment of the weakly log-concave case, without assuming an isoperimetric inequality, would require a specialized approach.
To the best of the authors' knowledge, a mixing time upper bound determined by the parameters $(\varepsilon; d,\warmstart{\pi^0},L(\pi))$ has not been proven; hence \Cref{assump:mixing_time_comparison} could not be satisfied by any practical example.
Nevertheless, a strategy that one could take is to construct a strongly log-concave approximation to the density \cite[Section 4.3]{dalalyan_theoretical_2017}.
Indeed, for such a construction, mixing time upper bounds have been derived \cite[Corollary 4]{dwivedi_log-concave_2019}, and a similar strategy as employed above could be used.
\section{\MarCo Dominates \OneblockMCMC in Peskun Ordering}
\label{sec:theory_marco_vs_one_block}

In \Cref{sec:marco_vs_joint_conditioning}, we showed that marginalization can improve the regularity of the density and lead to more efficient sampling. In this section, we show that \MarCo can outperform \OneblockMCMC, which is the closest \Joint strategy to \MarCo.
The advantage of \MarCo over \OneblockMCMC can be formalized by comparing the transition kernels, $k_{\Mtt}$ in \eqref{eq:marco_transition_kernel} and $k_{\Btt}$ in \eqref{eq:oneblock_transition_kernel}, respectively.  
Furthermore, the \MarCo advantage holds for any joint distribution, not just log-concave cases. 
The comparison of \OneblockMCMC and \MarCo transition kernels is natural because of the matching marginal proposal mechanism, $q_{\rm marg}$, and MH acceptance and rejection probabilities, $\alpha_{\rm marg}$ and $r_{\rm marg}$ in \eqref{eq:marco_acceptance_probability} and  \eqref{eq:marco_transition_kernel_rejection}, respectively. 
The difference between the two algorithms lies in their behavior under rejection. 
While \OneblockMCMC rejects the whole proposal $(\bfx^\star,\bfy^\star)$ and stays at $(\bfx^t,\bfy^t)$, \MarCo only rejects $\bfy^\star$ (staying at $\bfy^t$) and moves the $\bfx$-coordinate from $\bfx^t$ to $\bfx^{t+1} \sim \pcond(\cdot \mid \bfy^t)$ (see Figure~\ref{fig:chain_comparison}).
Intuitively, the movement of the \MarCo sampler upon rejection promotes mixing and improves the speed of convergence.

We formally validate this intuition using the framework of Peskun-Tierney ordering (see \cite{peskun_optimum_1973, tierney_note_1998, mira_ordering_2001, andrieu_peskuntierney_2021, andrieu_establishing_2016}). Let $\pi(\d\bfz)\in \Pcal(\Rbb^d)$ and let $k_1(\bfz,\d\hat\bfz)$ and $k_2(\bfz,\d\hat\bfz)$ be two $\pi$-reversible Markov transition kernels. We say that $k_1$ \textit{dominates} $k_2$ \textit{in the sense of Peskun}, denoted $k_1 \succeq k_2$, if
\begin{equation*}
    k_1(\bfz,A\setminus \{\bfz\}) \ge k_2(\bfz,A\setminus \{\bfz\}),\qquad\text{ for all }\bfz\in\mathbb{R}^d \text{ and } A\in\mathcal{B}(\mathbb{R}^d).
\end{equation*}
Intuitively, this condition means that the kernel $k_1$ has a higher probability of moving away from the current state $\bfz$ than the kernel $k_2$ does. Consequently, the corresponding chain explores the state space more effectively, leading to improved mixing properties.

We now show that the \MarCo transition kernel $k_{\Mtt}$ dominates the \OneblockMCMC transition kernel $k_{\Btt}$ in the sense of Peskun ($k_{\Mtt}\succeq k_{\Btt}$). Note that both  kernels are $\pjoint$-reversible because, as MH algorithms, both satisfy detailed balance with respect to $\pjoint$.

\begin{proposition}[$k_{\Mtt} \succeq k_{\Btt}$]
\label{prop:Peskun_ordering_M_B}
    For the same marginal proposal distribution, $q_{\rm marg}$, 
    the \MarCo transition kernel, $k_{\Mtt}$, dominates the \OneblockMCMC transition kernel, $k_{\Btt}$, in the sense of Peskun.
\end{proposition}

\begin{proof}
    Following the transition kernel definitions in \eqref{eq:marco_transition_kernel} and \eqref{eq:oneblock_transition_kernel},  for any $(\bfx,\bfy)\in\mathbb{R}^{\nx+\ny}$ and any $A\in\mathcal{B}(\mathbb{R}^{\nx+\ny})$, we have
    \begin{align*}
        k_{\Mtt}((\bfx,\bfy),A\setminus\{(\bfx,\bfy)\})&=\int_A \pcond(\hat\bfx\mid\hat\bfy)q_{\rm marg}(\hat\bfy\mid\bfy)\alpha_{\rm marg}(\bfy,\hat\bfy)\d\hat\bfx \d\hat\bfy \\
        &\hspace{3cm}+ \int_A \pcond(\hat\bfx\mid\hat\bfy)r_{\rm marg}(\bfy)\d\hat\bfx\delta_{\bfy}(d\hat\bfy),\\
        k_{\Btt}((\bfx,\bfy),A\setminus\{(\bfx,\bfy)\})&=\int_A \pcond(\hat\bfx\mid\hat\bfy)q_{\rm marg}(\hat\bfy\mid\bfy)\alpha_{\rm marg}(\bfy,\hat\bfy)\d\hat\bfx \d\hat\bfy.
    \end{align*}
    The Peskun ordering follows immediately because the second term in the expression for $ k_{\Mtt}((\bfx,\bfy),A\setminus\{(\bfx,\bfy)\})$ is non-negative.
\end{proof}

A Peskun-Tierney ordering yields a comparison of the efficiency of the corresponding Markov chains, that we can express in terms of asymptotic variance \eqref{eq:asymptotic_variance} and right spectral gap \eqref{eq:spectral_gap}. In the following theorem, we formalize such a comparison between \MarCo and \OneblockMCMC.

\begin{theorem}
For the same marginal proposal distribution, $q_{\rm marg}$, 
\begin{align*}
    \operatorname{Gap}_R(k_{\Mtt})\ge\operatorname{Gap}_R(k_{\Btt}) \qquad\text{and}\qquad \operatorname{var}(f,k_{\Mtt})\le\operatorname{var}(f,k_{\Btt})
\end{align*}
for any $\pi_{\rm joint}$-square-integrable function, $f$. 
\end{theorem}

\begin{proof}
    The claim follows directly from~\Cref{prop:Peskun_ordering_M_B} and~\cite[Theorem 2]{andrieu_establishing_2016}.
\end{proof}

As discussed in \Cref{sec:convergence_markov_chains}, larger right spectral gap and smaller asymptotic variance of \MarCo are associated with superior MCMC convergence properties compared to \OneblockMCMC in terms of convergence speed and Monte Carlo error.

\section{Numerical Examples}
\label{sec:numerics}

We now turn from theoretical analysis to numerical examples.
The first numerical example is described in \Cref{sec:deblurring}, where we investigate the practicality of \MarCo for a realistic semi-blind Bayesian image-deblurring problem.
For this example, the unknowns include both a high-dimensional image and a small number of parameters that define the blurring function.
Then, in \Cref{sec:funnel}, we explore a canonical Neal's funnel model to systematically test how \MarCo's performance scales with dimension and geometric difficulty of the target distribution.
We note that while neither experiment satisfies  strong log-concavity or isoperimetric inequality assumptions, empirical mixing time behaviors of such settings have been studied extensively in the literature \cite{dwivedi_log-concave_2019,wu_minimax_2022}.
Both examples demonstrate the improved MCMC convergence behavior of \MarCo compared to \Joint and \OneblockMCMC. 

We begin with a brief description of the diagnostics and implementation details.
In the experiments below, we report split-$\hat R$, effective sample size (ESS) and ESS per second, and Monte Carlo standard error (MCSE) when comparing posterior means. Together they distinguish evidence of stationarity, sampling efficiency after stationarity, and the resulting precision of posterior-mean estimates. We use split-$\hat R$ to assess whether multiple chains started from dispersed initializations have reached the same stationary regime, thereby providing a means to detect non-stationarity in the individual chains. A value of split-$\hat R$ closer to $1$ is better, as it indicates that the chains are exploring the same distribution (i.e., have mixed). To assess sampling efficiency, we provide an estimate of the ESS, which is the number of independent target draws that would yield the same asymptotic variance as the correlated MCMC draws. Thus, larger values are better. For computational comparisons, we present ESS per second.
Lastly, we consider the MCSE, which provides information about the uncertainty or error in resulting estimates. Smaller MCSE values are better. See \cite{carpenter_stan_2017,craiu_handbook_2026,gelman_bayesian_2013,vehtari_rank-normalization_2021} for formal definitions and further details.

All numerical experiments were implemented in Python using JAX \cite{bradbury_jax_2018} for array computations, automatic differentiation, just-in-time compilation, and explicit pseudorandom-number management.
MCMC transition kernels were built with BlackJAX, a composable Bayesian inference library for JAX \cite{cabezas_blackjax_2024}.
Code for the numerical experiments is available on GitHub at the repository \url{https://github.com/abhijit-c/marco}.

To provide fair comparisons of \MarCo with \Joint and \OneblockMCMC, our implementation leverages the deterministic pseudorandom number generating (PRNG) key of the JAX library~\cite{bradbury_jax_2018}.
Specifically, the proposal mechanisms for all MCMC algorithms are provided with the same PRNG keys for a total of $(T+1)M$ keys where $T + 1$ is the number of steps per chain and $M$ is the number of chains.
As a result, the $\bfy$-chains of \MarCo and \OneblockMCMC are identical because the proposal mechanism and acceptance probabilities are identical (see~\Cref{sec:marco_comparisons}).
However, contrary to \OneblockMCMC, \MarCo is designed to exploit parallelization in the sampling from the conditional distribution (see lines 9-10 in \Cref{alg:MarCo}).

\subsection{Example 1: Bayesian Semi-Blind Image Deblurring}
\label{sec:deblurring}

In semi-blind Bayesian image deblurring, we seek to infer a sharp high-dimensional image $\bfx\in\bbR^{\nx}$ from an observed blurred image $\bfd \in \bbR^{\nx}$, when only partial knowledge about the blur point spread function (PSF) is known, e.g., the PSF is defined by blur parameters $\bfy \in \bbR^{n_y}$ \cite{hansen_deblurring_2006}. 
For nonzero $\sigma_{e}$, assume that
\begin{equation}
\label{eq:semiblind_assumptions1}
\bfd = \bfB(\bfy)\bfx + \bfe, \quad \bfe\sim\calN(\bfzero,\sigma_{e}^{2}\bfI)
\end{equation}
where $\bfB(\bfy)$ belongs to a parametric class of circulant Gaussian blur operators with unnormalized kernel given by
$\exp \bigl(-\thf\bfr\t\bfSigma_b^{-1}\bfr\bigr)$ with $\bfSigma_b = \begin{psmallmatrix} s_1^2 & \rho s_1 s_2 \\ \rho s_1 s_2 & s_2^2 \end{psmallmatrix}$, $s_1,s_2>0$ and correlation $\rho\in(-1,1)$ for pixel offset $\bfr=[r_1, r_2]\t\in\bbZ^2$. We reparametrize, giving the blur parameters,
\begin{equation}
  \label{eq:deblur-theta}
  \bfy
  :=
  \begin{psmallmatrix}\log s_1 \\ \log s_2 \\ \tanh^{-1}\rho\end{psmallmatrix}
  \in\bbR^3,
  \qquad
  s_i = e^{\theta_i},\quad \rho = \tanh\theta_3,
\end{equation}
with Gaussian hyperprior $\bfy\sim \calN(\bfmu_y, \bfSigma_y)$.
For nonzero constant $\sigma_x$ and vector $\bfmu_x$, we assume an intrinsic Gaussian Markov random field prior with probability density function,
\begin{equation}
\label{eq:semiblind_assumptions2}
\pi(\bfx)\propto \exp\left(-\frac{1}{2 \sigma_x^2}(\bfx - \bfmu_x)\t \bfS\t \bfS (\bfx-\bfmu_x)\right),
\end{equation}
where $\bfS\t \bfS\in\bbR^{\nx\times\nx}$ is a convolution operator representing a 2D discrete Laplacian with periodic boundary conditions \cite{bardsley2018computational}.

The joint log-posterior density has the form
$\pjoint(\bfx, \bfy) = \exp(-V_{\rm joint}(\bfx, \bfy))$ where the joint potential is
\begin{equation}
  \label{eq:deblur-posterior}
 V_{\rm joint}(\bfx,\bfy)  =
 \thf \sigma_{e}^{-2} \| \bfB(\bfy)\bfx - \bfd\|^2 +  \thf \sigma_x^{-2} \| \bfS(\bfx - \bfmu_x)\|^2  + \thf \bigl\|\bfy - \bfmu_y \bigr\|_{\bfSigma_y^{-1}}^2 + c,
\end{equation}
where $c$ is a constant.
The conditional $\pcond(\bfx \mid \bfy, \bfd)$ is a high-dimensional Gaussian, and
the marginal is given by
\begin{equation}
\label{eq:deblur_marg}
\pmarg(\bfy\mid\bfd)  \mathrel{\propto}
\det\!\bigl(\bfH(\bfy)\bigr)^{-1/2}
    \exp\!\left(
    \thf\bfb(\bfy)\t\bfH(\bfy)^{-1}\bfb(\bfy)
    - \thf \bigl\|\bfy - \bfmu_y \bigr\|_{\bfSigma_y^{-1}}^2\right),
\end{equation}
where
$\bfH(\bfy) = \sigma_e^{-2}
    \bfB(\bfy)\t\bfB(\bfy)
    +\sigma_x^{-2}\bfS\t\bfS$ and $    \bfb(\bfy)=\sigma_e^{-2}\bfB(\bfy)\t\bfd+\sigma_x^{-2} \bfS\t \bfS \bfmu_x$.
Both $\bfB(\bfy)$ and $\bfS\t \bfS$ can be diagonalized using the discrete Fourier transform \cite{hansen_deblurring_2006}, so obtaining exact samples from the conditional and evaluating an unnormalized marginal density can be done efficiently, see \Cref{app:deblur-spectral}.
This is an ideal scenario for \MarCo to sample efficiently from $\pjoint$.%

We run 8 chains of 4000 warmup and 4000 post-warmup draws on a $128\times128$ pixel image ($\nx = 16{,}384$) with true blur parameters $s_1^\star = 1.2$, $s_2^\star = 2.4$, $\rho^\star = 0.35$ and noise level $\sigma_{e} = 0.01$.
We impose $\bfmu_y = (\log 1.8,\log 1.8,0)$ and $ \bfSigma_y = \operatorname{diag}(0.35^2,0.35^2,0.45^2)$ along with $\bfmu_x = \bfzero$ and $\sigma_x = 0.0625$.
We compare \Joint (MALA on $\pjoint$), \OneblockMCMC, and \MarCo (MALA on $\pmarg$ followed by exact conditional draws $\pcond$).
Convergence diagnostics for each of the blur parameters are provided in \Cref{tab:deblur-joint-vs-marginal} for \Joint and \MarCo.  ESS/s uses only the time spent sampling the corresponding joint or marginal chain, excluding conditional image reconstruction, thus isolating the effect of marginalizing the image on the blur-parameter chain.  We observe that for all blur parameters, the \MarCo values are better than \Joint (i.e.,  $\hat R$ closer to $1$ and larger $\ESS$ and ESS/s values).
The violin summaries in \Cref{fig:deblurring-ess-mcse-violin} show that \OneblockMCMC substantially improves pixel-wise ESS and posterior-mean MCSE over \Joint, while \MarCo produces ESS near the maximum of $32{,}000$ retained draws and the smallest MCSE.
They make the orders-of-magnitude separation between the samplers explicit while retaining the distribution across all pixels.

\begingroup
\pgfplotstableread[col sep=comma]
  {figures/deblurring/blur_parameter_metrics.csv}\DeblurringParameterMetrics
\pgfplotstableread[col sep=comma]
  {figures/deblurring/sampler_timings.csv}\DeblurringSamplerTimings
\newcommand{\DeblurringRhat}[1]{%
  \pgfplotstablegetelem{#1}{rhat}\of\DeblurringParameterMetrics
  \pgfmathprintnumber[fixed,precision=3]{\pgfplotsretval}%
}
\newcommand{\DeblurringESS}[1]{%
  \pgfplotstablegetelem{#1}{ess}\of\DeblurringParameterMetrics
  \pgfmathprintnumber[fixed,precision=0,1000 sep={,}]{\pgfplotsretval}%
}
\newcommand{\DeblurringESSRate}[2]{%
  \pgfplotstablegetelem{#1}{ess}\of\DeblurringParameterMetrics
  \edef\DeblurringESSValue{\pgfplotsretval}%
  \pgfplotstablegetelem{#2}{sampling_seconds}\of\DeblurringSamplerTimings
  \edef\DeblurringTimeValue{\pgfplotsretval}%
  \pgfmathparse{\DeblurringESSValue/\DeblurringTimeValue}%
  \pgfmathprintnumber[fixed,precision=2]{\pgfmathresult}%
}
\begin{table}[t]
  \centering
  \caption{Blur-parameter convergence diagnostics for \Joint and the \MarCo marginal chain. }
  \label{tab:deblur-joint-vs-marginal}
  \small
  \setlength{\tabcolsep}{5pt}
  \begin{tabular}{@{}l r r r r r r@{}}
    \toprule
    & \multicolumn{3}{c}{\Joint}
    & \multicolumn{3}{c}{\MarCo} \\
    \cmidrule(lr){2-4}\cmidrule(lr){5-7}
    & $\hat R$ & $\ESS$ & ESS/s
    & $\hat R$ & $\ESS$ & ESS/s \\
    \midrule
    $\log s_1$
      & \DeblurringRhat{0}$^\dagger$ & \DeblurringESS{0} & \DeblurringESSRate{0}{0}
      & \DeblurringRhat{3} & \DeblurringESS{3} & \DeblurringESSRate{3}{1} \\
    $\log s_2$
      & \DeblurringRhat{1}$^\dagger$ & \DeblurringESS{1} & \DeblurringESSRate{1}{0}
      & \DeblurringRhat{4} & \DeblurringESS{4} & \DeblurringESSRate{4}{1} \\
    $\tanh^{-1}\rho$
      & \DeblurringRhat{2}$^\dagger$ & \DeblurringESS{2} & \DeblurringESSRate{2}{0}
      & \DeblurringRhat{5} & \DeblurringESS{5} & \DeblurringESSRate{5}{1} \\
    \bottomrule
  \end{tabular}
  \par\footnotesize\raggedright
  $^\dagger$ indicates $\hat R>1.01$.
\end{table}

\newcommand{\DeblurringTiming}[2]{%
  \pgfplotstablegetelem{#1}{#2}\of\DeblurringSamplerTimings
  \pgfmathprintnumber[fixed,precision=1]{\pgfplotsretval}%
}

\begin{figure}[t]
  \centering
  \includegraphics[width=\linewidth]{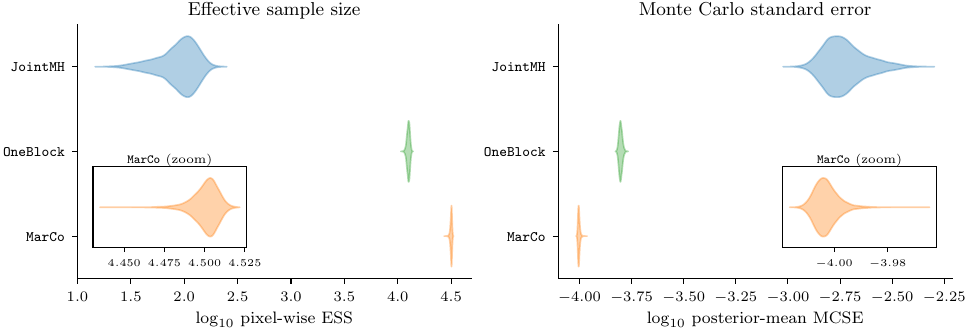}
  \caption{Distributions of pixel-wise ESS and
    posterior-mean MCSE across all
    $16{,}384$ pixels. Violins are computed in $\log_{10}$; insets
    magnify the tightly concentrated \MarCo distributions.}
  \label{fig:deblurring-ess-mcse-violin}
\end{figure}

Taken together, the blur-parameter and image diagnostics demonstrate that \OneblockMCMC and \MarCo both outperform \Joint for sampling from the joint posterior. Moreover,
compared to \OneblockMCMC, \MarCo achieves approximately $2.5\times$ the pixel-wise ESS and $3.5\times$ the pixel-wise ESS/s speedups for sampling from the conditional. Finally, we provide a comparison of CPU times (in seconds) for post-warmup sampling. Compared to \Joint which took 43.2 seconds, \MarCo took 37.8 (=20.9+16.9) seconds and \OneblockMCMC took 52.5 (=21+31.5) seconds, where the first number denotes the time for sampling from the marginal and the second number denotes the time for sampling from the conditional. It is worth noting that the additional time comes from the conditional sampling (31.5 seconds for \OneblockMCMC versus 16.9 seconds with \MarCo). We note that with higher parallelism, \MarCo can perform even better than the stated ESS/s and runtime.
\endgroup
\subsection{Example 2: Neal's Funnel}
\label{sec:funnel}

Sampling from Neal's funnel~\cite{neal_slice_2003}, a distribution that arises in Bayesian inference (e.g., ~\cite[Section 4.2]{zhang_transport_2025}), is challenging due to the exponentially-varying geometry.
Formally, we consider a hierarchical model where the marginal variable $y\in \Rbb$ is sampled from a univariate Gaussian and, conditional on $y$, the variable $\bfx \in \Rbb^\nx$ is sampled from a multivariate Gaussian whose variance depends exponentially on $y$; i.e.,
\begin{equation}
  \label{eq:funnel-hierarchy}
  y \sim \calN(0,\sigma^2)
  \quad \text{and}\quad
  \bfx \mid y \sim \calN(\bfzero, e^y\bfI_\nx)
\end{equation}
where $\sigma > 0$ is the standard deviation of the marginal distribution.
The exponential scale on the variance of the conditional variable creates a rapid narrowing of the density along the $y$-axis; visually, the density looks like a funnel (see \Cref{fig:devils_funnel}).
Because the marginal and conditional densities are Gaussian in \eqref{eq:funnel-hierarchy}, we have $\pjoint(\bfx, y) = \exp(-V_{\rm joint}(\bfx, y))$ where the joint potential is
\begin{equation}\label{eq:joint_pontential_funnel}
  V_{\rm joint}(\bfx,y)
  =
  \tfrac{1}{2} \sigma^{-2}y^2
  -\tfrac{1}{2} \nx y
  +\tfrac{1}{2} e^{-y}\|\bfx\|^2 - \log(\sigma\sqrt{(2\pi)^{\nx+1}})
\end{equation}
and the marginal potential \eqref{def:Jmarg} is $V_{\rm marg}(y) = \tfrac{1}{2} \sigma^{-2}y^2 - \log(\sigma\sqrt{2\pi})$.
From the Hessian
\begin{equation}
  \label{eq:funnel-hessian}
  \nabla^2 V_{\rm joint}(\bfx,y)
  =
  \begin{bmatrix}
    e^{-y}\bfI_\nx & -e^{-y}\bfx \\
    -e^{-y}\bfx\t & \sigma^{-2}
      + \frac{1}{2}e^{-y}\|\bfx\|^2
  \end{bmatrix},
\end{equation}
one can verify that the condition number $\kappa(\pjoint)$ depends exponentially on $y$.
For example, evaluating the Hessian at $\bfx = {\bf0}_{\nx}$ results in diagonal matrix, $\nabla V_{\rm joint}({\bf0}_{\nx}, y)$, with diagonal entries equal to $e^{-y}$ or $\sigma^{-2}$.
Thus, the condition number of $\pjoint$ must be at least $\max \{\sigma^2 e^{-y}, \sigma^{-2} e^y\}$.
Notice that as $y \to \pm \infty$, the condition number grows to $\infty$, indicating that $\pjoint$ is a weakly log concave distribution.
In comparison, because $\pmarg$ is a univariate Gaussian distribution, it is strongly log concave and $\kappa(\pmarg) = 1$.
Thus, one would expect MCMC algorithms for $\pmarg$ to mix faster than MCMC for $\pjoint$.
\def\h{0.5}

\begin{figure}[!htb]
\centering

\begin{subfigure}{\linewidth}
\centering
\includegraphics[width=\linewidth]{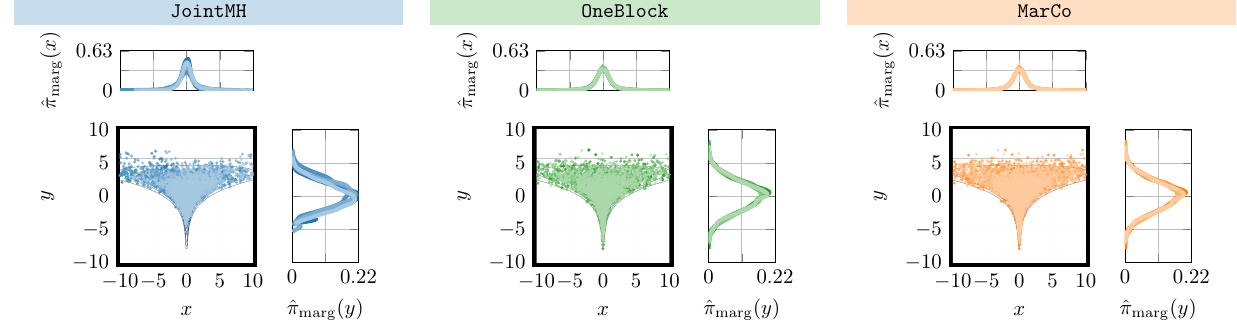}
\caption{Visualizations of samples drawn from 2D Neal's funnel from $4$ of the $8$ Markov chains.
For each algorithm, the central 2D plot shows the samples generated overlaying a contour plot of the 2D Neal's funnel.
Empirical marginal distributions for each chain appear above, $\hat{\pi}_{\rm marg}(x)$, and to the right, $\hat{\pi}_{\rm marg}(y)$, of the central plot.}
\label{fig:devils_funnel_samples}
\end{subfigure}

\begin{subfigure}{\linewidth}
\centering

\def\resultsDir{figures/toy_examples/devils_funnel}
\pgfplotstableread[col sep=comma]{\resultsDir/joint/metrics.csv}\loadedtable
\pgfplotstableread[col sep=comma]{\resultsDir/oneblock/metrics.csv}\tableoneblock
\pgfplotstablevertcat{\loadedtable}{\tableoneblock}
\pgfplotstableread[col sep=comma]{\resultsDir/marco/metrics.csv}\tablemarco
\pgfplotstablevertcat{\loadedtable}{\tablemarco}

\pgfplotstabletypeset[
	metricstablestyle,
	columns={samplername, x_rhat, y_rhat, x_ess, y_ess, x_ess_per_sec, y_ess_per_sec},
	columns/x_ess/.style={column type={@{\hspace{\h cm}}r}, fixed, precision=0, zerofill, column name={\multicolumn{1}{c}{$x$}}},
	columns/y_ess/.style={column type={r}, fixed, precision=0, zerofill, column name={\multicolumn{1}{c}{$y$}}},
	columns/x_ess_per_sec/.style={column type={@{\hspace{\h cm}}r}, fixed, precision=2, zerofill, column name={\multicolumn{1}{c}{$x$}}},
	columns/y_ess_per_sec/.style={column type={r}, fixed, precision=2, zerofill, column name={\multicolumn{1}{c}{$y$}}},
	columns/x_rhat/.style={column type={@{\hspace{\h cm}}r}, fixed, precision=4, zerofill, column name={\multicolumn{1}{c}{$x$}}},
	columns/y_rhat/.style={column type={r}, fixed, precision=4, zerofill, column name={\multicolumn{1}{c}{$y$}}},
    	every head row/.style={output empty row},
	every head row/.style={
        before row={%
            \toprule
            &
            \multicolumn{2}{c}{split-$\hat{R}$ ($\approx 1$)}
            &
            \multicolumn{2}{c}{ESS ($\uparrow$)} & \multicolumn{2}{c}{ESS/s ($\uparrow$)} \\
            \cmidrule(lr){2-3}  \cmidrule(lr){4-5} \cmidrule(lr){6-7}
        },
        after row=\midrule, %
    },
        every row 2 column 1/.style={postproc cell content/.append style={/pgfplots/table/@cell content/.add={\cellcolor{lightgray!50} \boldmath}{},}},
        every row 2 column 2/.style={postproc cell content/.append style={/pgfplots/table/@cell content/.add={\cellcolor{lightgray!50} \boldmath}{},}},
        every row 2 column 3/.style={postproc cell content/.append style={/pgfplots/table/@cell content/.add={\cellcolor{lightgray!50} \boldmath}{},}},
        every row 2 column 4/.style={postproc cell content/.append style={/pgfplots/table/@cell content/.add={\cellcolor{lightgray!50} \boldmath}{},}},
        every row 2 column 5/.style={postproc cell content/.append style={/pgfplots/table/@cell content/.add={\cellcolor{lightgray!50} \boldmath}{},}},
        every row 2 column 6/.style={postproc cell content/.append style={/pgfplots/table/@cell content/.add={\cellcolor{lightgray!50} \boldmath}{},}},
        every row 0 column 1/.style={postproc cell content/.append style={/pgfplots/table/@cell content/.add={}{}${}^\dagger$,}},
        every row 0 column 2/.style={postproc cell content/.append style={/pgfplots/table/@cell content/.add={}{}${}^\dagger$,}},
        every row 0 column 3/.style={postproc cell content/.append style={/pgfplots/table/@cell content/.add={}{}${}^\dagger$,}},
        every row 0 column 4/.style={postproc cell content/.append style={/pgfplots/table/@cell content/.add={}{}${}^\dagger$,}},
        every row 0 column 5/.style={postproc cell content/.append style={/pgfplots/table/@cell content/.add={}{}${}^\dagger$,}},
        every row 0 column 6/.style={postproc cell content/.append style={/pgfplots/table/@cell content/.add={}{}${}^\dagger$,}},
	]\loadedtable

\caption{Mixing diagnostics of the three algorithms.
By design, \OneblockMCMC and \MarCo have identical ESS for the $y$-chain.
\MarCo mixes faster in the $x$-chain (\Cref{sec:theory_marco_vs_one_block}) and parallelizes the draws from the $x$-chain, resulting in the highest ESS/s. 
For \Joint, $\dagger$ denotes the trial failed to satisfy $\text{split-$\hat R$} < 1.01$. Hence, the ESS values should be interpreted with caution.}
\label{tab:devils_funnel_metrics}
\end{subfigure}

\caption{2D Neal's funnel demonstration using $\sigma = 3$ for the marginal standard deviation.}
\label{fig:devils_funnel}
\end{figure}

We compare \Joint, \OneblockMCMC, and \MarCo on Neal's funnel examples of various dimensions of the joint, $d$, and standard deviations $\sigma$.
For all algorithms, we draw samples using {\tt MALA}  with $8$ chains and generate $4000$ samples per chain.
In order to properly tune the step size (\Cref{sec:spectral_consequences}), we warm-up for $4000$ steps with an initial step size of $10^{-2}$ and a target acceptance rate of $0.574$.
We tune the step size using dual-averaging adaptation \cite[Section 3.2.1]{hoffman_no-u-turn_2014}.
Before warm up, each chain is initialized randomly and the same initial $y$-states are used for all three algorithms.
We run each algorithm over $5$ random initializations.

The benefits of mixing in the low-dimensional, well-conditioned marginal space are verified numerically in \Cref{fig:devils_funnel}.
As expected, \Joint exhibits the weakest diagnostics in \Cref{tab:devils_funnel_metrics}, including split-$\hat{R}$ values greater than BlackJAX-suggested threshold of $1.01$, indicating a lack of convergence.
In comparison, the split-$\hat R$ values of the marginal MCMC algorithms, \OneblockMCMC and \MarCo, indicate substantially better agreement across chains.
Moreover, the ESS and ESS per second values of \OneblockMCMC and \MarCo are two orders of magnitude greater than those of \Joint, supporting the theoretical results in \Cref{sec:spectral_consequences} that MCMC algorithms applied to the lower-dimensional, better-conditioned marginal exhibit faster mixing times. 
Comparing the marginalized methods, we see that the \MarCo achieves about twice as large an $\bfx$-chain ESS as \OneblockMCMC, reflecting the benefits of moving upon rejection. 
We see the additional benefits of parallelized sampling of the $\bfx$-chain in the ESS/s metric, where \MarCo achieves a 3-4 times speed up. 

We further explore the effects of the geometry of Neal's funnel on MCMC performance, varying both dimension and marginal standard deviation
The diagnostics in \Cref{tab:funnel_scaling_min_ess_x} reveal that \Joint becomes fragile as either source of funnel difficulty increases (higher dimension or greater marginal variance).
For the same budget across algorithms, \Joint failed to achieve satisfactory cross-chain mixing in the more difficult funnel configurations.
In contrast, both \MarCo and \OneblockMCMC samplers sustain stronger mixing metrics as the dimension and funnel scale increase by marginalizing.  
Across all parameters, \MarCo\ consistently performs best, consistent with the theoretical insights in \Cref{sec:spectral_consequences} and \Cref{sec:theory_marco_vs_one_block}. 

\begin{table}[!htb]
\centering
\scriptsize

\pgfplotstableread[col sep=comma]
{figures/funnel_scaling/generated/min_ess_x_sigma_1.csv}\loadedtablefunnel
\pgfplotstableread[col sep=comma]
{figures/funnel_scaling/generated/min_ess_x_sigma_2.csv}\funnelsigmatwo
\pgfplotstableread[col sep=comma]
{figures/funnel_scaling/generated/min_ess_x_sigma_3.csv}\funnelsigmathree

\pgfplotstablecreatecol[copy column from table={\funnelsigmatwo}{d4}]{d4_2}{\loadedtablefunnel}
\pgfplotstablecreatecol[copy column from table={\funnelsigmatwo}{d16}]{d16_2}{\loadedtablefunnel}
\pgfplotstablecreatecol[copy column from table={\funnelsigmatwo}{d64}]{d64_2}{\loadedtablefunnel}
\pgfplotstablecreatecol[copy column from table={\funnelsigmatwo}{d256}]{d256_2}{\loadedtablefunnel}

\pgfplotstablecreatecol[copy column from table={\funnelsigmathree}{d4}]{d4_3}{\loadedtablefunnel}
\pgfplotstablecreatecol[copy column from table={\funnelsigmathree}{d16}]{d16_3}{\loadedtablefunnel}
\pgfplotstablecreatecol[copy column from table={\funnelsigmathree}{d64}]{d64_3}{\loadedtablefunnel}
\pgfplotstablecreatecol[copy column from table={\funnelsigmathree}{d256}]{d256_3}{\loadedtablefunnel}

\pgfplotstableset{%
    highlightrow/.style={
        postproc cell content/.append code={
            \count1=\pgfplotstablerow%
            \advance\count1 by1%
            \ifnum\count1=#1%
                \ifnum\pgfplotstablecol>0\relax
                    \pgfkeysalso{@cell content/.add={\ifmmode\else\cellcolor{lightgray!50} \boldmath\fi$}{$}}%
                \fi%
            \fi%
        },
    },
}

{
\setlength{\tabcolsep}{1.5pt}
\def\h{0.15}
\pgfplotstabletypeset[
	metricstablestyle,
	columns={samplername, 
        d4, d16, d64, d256, 
        d4_2, d16_2, d64_2, d256_2, 
        d4_3, d16_3, d64_3, d256_3},
	columns/d4/.style={column type={r}, fixed, precision=0, zerofill, column type={@{\hspace{\h cm}}r}, column name={\multicolumn{1}{c}{$d=4$}}},
	columns/d16/.style={column type={r}, fixed, precision=0, zerofill, column type={r},, column name={\multicolumn{1}{c}{$16$}}},
	columns/d64/.style={column type={r}, fixed, precision=0, zerofill, column type={r}, column name={\multicolumn{1}{c}{$64$}}},
	columns/d256/.style={column type={r}, fixed, precision=0, zerofill, column type={r}, column name={\multicolumn{1}{c}{$256$}}},
	columns/d4_2/.style={column type={@{\hspace{\h cm}}r}, fixed, precision=0, zerofill, column name={\multicolumn{1}{c}{$d=4$}}},
	columns/d16_2/.style={column type={r}, fixed, precision=0, zerofill, column type={r}, column name={\multicolumn{1}{c}{$16$}}},
	columns/d64_2/.style={column type={r}, fixed, precision=0, zerofill, column type={r}, column name={\multicolumn{1}{c}{$64$}}},
	columns/d256_2/.style={column type={r}, fixed, precision=0, zerofill, column type={r}, column name={\multicolumn{1}{c}{$256$}}},
	columns/d4_3/.style={column type={@{\hspace*{\h cm}}r}, fixed, precision=0, zerofill, column name={\multicolumn{1}{c}{$d=4$}}},
	columns/d16_3/.style={column type={r}, fixed, precision=0, zerofill, column type={r},, column name={\multicolumn{1}{c}{$16$}}},
	columns/d64_3/.style={column type={r}, fixed, precision=0, zerofill, column type={r}, column name={\multicolumn{1}{c}{$64$}}},
	columns/d256_3/.style={column type={r}, fixed, precision=0, zerofill, column type={r}, column name={\multicolumn{1}{c}{$256$}}},
    every head row/.style={output empty row},
	every head row/.style={
        before row={%
            \toprule
            &
            \multicolumn{4}{c}{$\sigma=1$}
            & 
            \multicolumn{4}{c}{$\sigma=2$}
            &
            \multicolumn{4}{c}{$\sigma=3$} \\
            \cmidrule(lr){2-5}  \cmidrule(lr){6-9} \cmidrule(lr){10-13}
        },
        after row=\midrule, %
    },
        highlightrow = {2},
        every row 0 column 3/.style={postproc cell content/.append style={/pgfplots/table/@cell content/.add={}{}${}^\dagger$,}},
        every row 0 column 4/.style={postproc cell content/.append style={/pgfplots/table/@cell content/.add={}{}${}^\dagger$,}},
        every row 0 column 5/.style={postproc cell content/.append style={/pgfplots/table/@cell content/.add={}{}${}^\dagger$,}},
        every row 0 column 6/.style={postproc cell content/.append style={/pgfplots/table/@cell content/.add={}{}${}^\dagger$,}},
        every row 0 column 7/.style={postproc cell content/.append style={/pgfplots/table/@cell content/.add={}{}${}^\dagger$,}},
        every row 0 column 8/.style={postproc cell content/.append style={/pgfplots/table/@cell content/.add={}{}${}^\dagger$,}},
        every row 0 column 9/.style={postproc cell content/.append style={/pgfplots/table/@cell content/.add={}{}${}^\dagger$,}},
        every row 0 column 10/.style={postproc cell content/.append style={/pgfplots/table/@cell content/.add={}{}${}^\dagger$,}},
        every row 0 column 11/.style={postproc cell content/.append style={/pgfplots/table/@cell content/.add={}{}${}^\dagger$,}},
	]\loadedtablefunnel
}

\caption{
    Minimum $\bfx$-ESS in the Neal's funnel scaling study, averaged over five random trials.
    Column blocks fix the funnel scale $\sigma$ and report conditional dimensions $d$ within each block.
    \MarCo obtains the highest ESS values, approximately twice as large as the ESS of \OneblockMCMC and 2-3 orders of magnitude larger than \Joint.
    For \Joint, $\dagger$ denotes that at least one of trials failed to satisfy $\text{split-$\hat R$} < 1.01$. Hence, ESS values should be interpreted with caution.
}
\label{tab:funnel_scaling_min_ess_x}
\end{table}

\section{Conclusions and Future Work}
\label{sec:conclusions}

We introduced \MarCo, a marginal-conditional strategy for sampling from a joint distribution.
\MarCo inherits the convergence properties of the MH sampler for the marginal and converges to the target distribution in total-variation. Due to sampling in a lower-dimensional and potentially better conditioned space, \MarCo has sharper mixing time upper bounds than \Joint.
Moreover, \MarCo dominates \OneblockMCMC in the Peskun ordering, yielding a no-smaller right spectral gap and no-larger asymptotic variance. The theoretical results demonstrate \MarCo's superior convergence properties. The numerical experiments support the theoretical results and demonstrate that, in practice, \MarCo performs significantly better than existing sampling methods.

An important direction for future work is to extend \MarCo to large-scale problems, where exact sampling from the conditional may not be feasible. Marginal potentials often require costly linear solvers, trace, or log-determinant evaluations; randomized low-rank and matrix-free estimators offer a promising route to reducing these costs \cite{saibaba_randomized_2017}. 
Combining such approximations with delayed acceptance or pseudomarginal corrections could extend \MarCo to settings where the marginal cannot be evaluated exactly while retaining the correct target distribution \cite{saibaba_efficient_2019}. 
Additional ideas from iterative sampling may be able to gradually improve the accuracy of the inexact conditional draws \cite{newman_slimtrain---stochastic_2022, chung_sampled_2020, slagel_sampled_2019}. 
On the theoretical side, sharper guarantees under weaker geometric assumptions, as well as bounds that account for randomized or inexact marginal evaluations, would 
extend the impact and practicality of \MarCo.
  
\section*{Acknowledgments}
The authors E. Newman and A. Chowdhary acknowledge support from the United States Air Force Office of Scientific Research under FA9550-26-1-B049 (program manager Dr. Fariba Fahroo).
This work was partially supported by the National Science Foundation under grants DMS-2411197 (J. Chung), DMS-2309751 (E. Newman), DMS-2541280 (E. Newman), and by the Knut and Alice Wallenberg Foundation under grant KAW 2024.0327 (F. Milinanni).
Moreover, this material is based upon work supported by the National Science Foundation under Grant No.~DMS-1929284 while the four authors were in residence at the Institute for Computational and Experimental Research in Mathematics in Providence, RI, during the Stochastic and Randomized Algorithms in Scientific Computing program.
The numerical experiments were performed through the hardware made available in the Tufts University High Performance Compute Cluster \cite{phimmasen_reference_2021}.

The authors also acknowledge the use of GPT-5.5, GPT-5.6 Sol, Gemini Pro 3.1, and Claude Sonnet 5 in writing the experiments, primarily for debugging and SLURM orchestration.
The same tools were also consulted during the general research and development of the theory presented in this manuscript, and for polishing the writing.
The authors take full responsibility for the accuracy and integrity of all content in this paper.

\bibliographystyle{siamplain}

\appendix
\section{Marginal Hessian Identity}
\label{appendix:marginal-hessian-identity}

To show the marginal Hessian identity in \Cref{lem:marginal-hessian-identity}, we first prove a minor differentiation-under-the-integral lemma for functions.
\begin{lemma}
    \label{lem:leibniz}
    Under the joint density regularity assumptions (\Cref{assump:joint_density_regularity}), the following differentiability-under-the-integral identities hold for every $\bfy\in\Rbb^\ny$:
    \begin{align}
        \label{eq:leibniz-pmarg}
        \nabla_{\bfy}\pmarg(\bfy)
        &= \int_{\Rbb^\nx}\nabla_{\bfy}\pjoint(\bfx,\bfy)\,\d\bfx, \\
        \label{eq:leibniz-conditional-grad}
        \nabla_{\bfy}\!\int_{\Rbb^\nx}\nabla_{\bfy} V(\bfx,\bfy)\,\pcond(\bfx\mid\bfy)\,\d\bfx
        &= \int_{\Rbb^\nx}\nabla_{\bfy}\!\bigl(\nabla_{\bfy} V(\bfx,\bfy)\,\pcond(\bfx\mid\bfy)\bigr)\,\d\bfx.
    \end{align}
\end{lemma}
\begin{proof}[Proof sketch]
Note, every log-concave density admits an exponential envelope $\pjoint(\bfz)\leq \exp(-a\|\bfz\|+b)$ for some $a>0$ and $b\in\bbR$ \cite[Theorem~5.1]{saumard_log-concavity_2014}.
Moreover, finite smoothness implies $\|\nabla V(\bfz)\|\leq C(1+\|\bfz\|)$ and $\|\nabla^2V(\bfz)\|\leq L(\pjoint)$.
Consequently, on every bounded set of $\bfy$-values, the integrands and their $\bfy$-derivatives in \eqref{eq:leibniz-pmarg}--\eqref{eq:leibniz-conditional-grad} are dominated by an integrable function of the form $C(1+\|\bfx\|^2)e^{-a\|\bfx\|}$.
The identities, therefore, follow by applying the differentiation-under-the-integral theorem \cite[Theorem~2.27]{folland_real_1999} coordinate-wise.
\end{proof}

\begin{proof}[Proof of~\Cref{lem:marginal-hessian-identity}]
	We first compute the gradient of $V_{\rm marg}(\bfy) = -\log \pmarg(\bfy)$ using \Cref{lem:leibniz}:
		\begin{align}
			\begin{split}
			\nabla V_{\rm marg}(\bfy)
				=-\frac{\nabla \pmarg(\bfy)}{\pmarg(\bfy)}
				= -\frac{\int_{\Rbb^\nx}  \nabla_{\bfy}\log \pjoint(\bfx, \bfy) \pjoint(\bfx, \bfy)\,\d\bfx}{\pmarg(\bfy)}.
			\end{split}
		\end{align}
		Using the conditional-marginal factorization~\eqref{eq:joint_factorization}, we get
		\begin{align}\label{eq:grad_V_marg_mean}
			\nabla V_{\rm marg}(\bfy)
				= -\int_{\Rbb^\nx}  \nabla_{\bfy}\log \pjoint(\bfx, \bfy) \cdot \pcond(\bfx \mid \bfy) \,\d\bfx
				= \Ebb_{\bfx\sim \pcond(\cdot \mid \bfy)}\left[ \nabla_{\bfy}V(\bfx, \bfy) \right]
		\end{align}
	where $V(\bfx, \bfy) = -\log \pjoint(\bfx, \bfy)$.
	Differentiating again, using \Cref{lem:leibniz}, we obtain
		\begin{align}\label{eq:hessian_Vmarg_step}
		\nabla^2 V_{\rm marg}(\bfy)
				&= \int_{\Rbb^\nx}  \nabla_{\bfy}^2 V(\bfx, \bfy) \pcond(\bfx \mid \bfy) \,\d\bfx +  \int_{\Rbb^\nx}  \nabla_{\bfy} V(\bfx, \bfy) \nabla_{\bfy}\pcond(\bfx \mid \bfy)^\top \,\d\bfx.
		\end{align}
	By differentiating through the log of the standard factorization~\eqref{eq:joint_factorization}, we see that
		\begin{align*}%
		\nabla_{\bfy} \log\pcond(\bfx \mid \bfy) = -\nabla_{\bfy} V(\bfx, \bfy) + \nabla V_{\rm marg}(\bfy),
		\end{align*}
	hence, recalling~\eqref{eq:grad_V_marg_mean}, we obtain
    \begin{equation}
    \label{eq:covariance_term}
    \begin{split}
        \int_{\Rbb^\nx} \nabla_{\bfy} V(\bfx, \bfy) \nabla_{\bfy}\pcond(\bfx \mid &\bfy)^\top\,\d\bfx=-\int_{\bbR^\nx}\nabla_\bfy V(\bfx,\bfy)\nabla_\bfy V(\bfx,\bfy)^\top \pcond(\bfx\mid\bfy)d\bfx\\
        &
        + \Ebb_{\bfx\sim \pcond(\cdot \mid \bfy)}\left[ \nabla_{\bfy}V(\bfx, \bfy) \right]\Ebb_{\bfx\sim \pcond(\cdot \mid \bfy)}\left[ \nabla_{\bfy}V(\bfx, \bfy) \right]^\top.
        \end{split}
    \end{equation}
    Substituting~\eqref{eq:covariance_term} into~\eqref{eq:hessian_Vmarg_step},  the marginal Hessian simplifies to the desired form.
\end{proof}

\section{Bayesian Image Deblurring: Efficient Spectral Evaluations}
\label{app:deblur-spectral}
We describe how elements of \MarCo can be computed efficiently for the image deblurring example by exploiting the eigendecompositions of $\bfB(\bfy)$ and $\bfS\t \bfS$, whereby they share the two-dimensional DFT eigenbasis $\bfF\in\bbC^{\nx\times\nx}$.  That is,
\begin{equation}
\label{eq:spectral}
   \bfB(\bfy) = \bfF^* \bfLambda_b(\bfy) \bfF \quad \mbox{and} \quad \bfS\t \bfS = \bfF^* \bfLambda_s \bfF,
\end{equation}
where diagonal matrices  $\bfLambda_b(\bfy)$ and $\bfLambda_s$ contain eigenvalues of $\bfB(\bfy)$ and $\bfS\t \bfS$ respectively \cite{hansen_deblurring_2006}.

The conditional is given by $\pcond(\bfx \mid \bfy, \bfd) \sim\calN(\bfmu_{\rm cond}(\bfy),\bfSigma_{\rm cond}(\bfy))$, where 
$\bfSigma_{\rm cond}(\bfy) = (\sigma_{e}^{-2}\bfB(\bfy)\t \bfB(\bfy) + \sigma_x^{-2}\bfS\t \bfS )^{-1} = \bfF^* (\sigma_{e}^{-2} |\bfLambda_b(\bfy)|^2 +\sigma_x^{-2} \bfLambda_s)^{-1} \bfF$, and 
the conditional mean can be computed in the Fourier space as
\begin{align}
    \bfmu_{\rm cond}(\bfy)&=\bfSigma_{\rm cond}(\bfy)(\sigma_{e}^{-2}\bfB(\bfy)\t \bfd + \sigma_x^{-2} \bfS\t \bfS \bfmu_x) \\
    &= \bfF^* (\sigma_{e}^{-2} |\bfLambda_b(\bfy)|^2 +\sigma_x^{-2} \bfLambda_s)^{-1}(\sigma_{e}^{-2}\overline{\bfLambda_b(\bfy)}\t \bfF \bfd + \sigma_x^{-2} \bfLambda_s \bfF \bfmu_x),
\end{align}
where $\overline{(\cdot)}$ denotes complex conjugate.
Moreover, an exact conditional sample can be obtained in $\bigO{\nx\log\nx}$ as $\bfmu_{\rm cond}(\bfy) + \bfSigma_{\rm cond}(\bfy)^{1/2} \bfz$
where $\bfz\sim\calN(\bfzero,\bfI_\nx)$.

The marginal \eqref{eq:deblur_marg} can also be evaluated efficiently using \eqref{eq:spectral}, since
\begin{equation}
    \bfH(\bfy) = \sigma_e^{-2}\bfB(\bfy)\t\bfB(\bfy)
    +\sigma_x^{-2}\bfS\t\bfS
    = \bfF^* ( \sigma_e^{-2} |\bfLambda_b(\bfy)|^2 + \sigma_x^{-2} \bfLambda_s) \bfF .
\end{equation}
This enables efficient computation of the terms in \eqref{eq:deblur_marg}, such that evaluating $V_{\rm marg}(\bfy)$ costs
$\bigO{\nx\log\nx}$ (dominated by the FFTs). Gradients can be obtained via the chain rule and incur no additional FFTs.

\end{document}